\documentclass[11pt]{article}
\usepackage[utf8]{inputenc}
\usepackage{amsmath,amsfonts,bm,amsthm}
\usepackage{amssymb,enumitem,mathtools,mathrsfs}
\usepackage{graphicx}
\usepackage{url}
\usepackage{appendix}
\usepackage{hyperref}

\makeatletter
\newcommand*{\transpose}{%
  {\mathpalette\@transpose{}}%
}
\newcommand*{\@transpose}[2]{%
  \raisebox{\depth}{$\m@th#1\intercal$}%
}

\usepackage{lipsum}

\let\OLDthebibliography\thebibliography
\renewcommand\thebibliography[1]{
  \OLDthebibliography{#1}
  \setlength{\parskip}{0pt}
  \setlength{\itemsep}{0pt plus 0.3ex}
}

\usepackage{array,etoolbox}
\preto\tabular{\setcounter{magicrownumbers}{0}}
\newcounter{magicrownumbers}

\preto\tabular{\setcounter{magicrownumbers2}{0}}
\newcounter{magicrownumbers2}

\newtheorem{theorem}{Theorem}[section]
\newtheorem{lemma}[theorem]{Lemma}
\newtheorem{proposition}[theorem]{Proposition}
\newtheorem{corollary}[theorem]{Corollary}
\theoremstyle{definition}

\newtheorem{example}[theorem]{Example}

\theoremstyle{remark}
\newtheorem{remark}[theorem]{Remark}

\newtheorem{question}[theorem]{Question}

\title{The coherent rank of a graph with three eigenvalues}
\author{Gary Greaves 
  \thanks{School of Physical and Mathematical Sciences, 
  Nanyang Technological University, 
  21 Nanyang Link, Singapore 637371,
 \tt{gary@ntu.edu.sg}}
 \and
 Jose Yip
  \thanks{School of Physical and Mathematical Sciences, 
  Nanyang Technological University, 
  21 Nanyang Link, Singapore 637371,
 \tt{josezhengho.yip@ntu.edu.sg}}
}
\date{}

\begin{document}

\maketitle

\begin{abstract}
    We characterise graphs that have three distinct eigenvalues and coherent ranks 8 and 9, linking the former to certain symmetric $2$-designs and the latter to specific quasi-symmetric $2$-designs. This characterisation leads to the discovery of a new biregular graph with three distinct eigenvalues. Additionally, we demonstrate that the coherent rank of a triregular graph with three distinct eigenvalues is at least 14. Finally, we introduce a conjecturally infinite family of biregular graphs with three distinct eigenvalues, obtained by switching the block graphs of orthogonal arrays.
\end{abstract}

\section{Introduction}

Strongly regular graphs are a fundamental object of study in algebraic graph theory. These graphs can be characterised by their spectrum: connected strongly regular graphs are precisely the connected regular graphs that have precisely three distinct eigenvalues. 
Thus, graphs with exactly three distinct eigenvalues can be viewed as a generalisation of strongly regular graphs.

At the British Combinatorial Conference in 1995, Willem Haemers posed a question: apart from complete bipartite graphs, which non-regular graphs have precisely three distinct eigenvalues? This prompted seminal work by Muzychuk and Klin~\cite{MUZYCHUK1998191} and Van Dam~\cite{VANDAM1998101}, who each produced infinite families of non-regular graphs with this property. These families include graphs with vertices having exactly two distinct degrees, as well as a finite number of examples with three distinct degrees, such as the multiplicative cones discovered by Bridges and Mena~\cite{BM81} in 1981.

Let $\mathscr{G}_3$ denote the set of all connected graphs having precisely three distinct eigenvalues. Several attempts have been made to systematically classify the graphs in $\mathscr{G}_3$:
\begin{itemize}
    \item Van Dam~\cite{VANDAM1998101}: graphs in $\mathscr{G}_3$ with at most 29 vertices;
    \item Chuang and Omidi~\cite{chuang}: graphs in $\mathscr{G}_3$ whose largest eigenvalue is less than 8;
    \item Van Dam~\cite{VANDAM1998101} again: graphs in $\mathscr{G}_3$ whose smallest eigenvalue is at least $-2$;
    \item Cheng et al.~\cite{Cheng18}: graphs in $\mathscr{G}_3$ whose second largest eigenvalue is at most 1.
\end{itemize}
These classifications did not yield any graphs in $\mathscr{G}_3$ with more than three distinct degrees. 
In fact, the following question by De Caen~\cite{ddC99} and {\cite[Problem 9]{vDddC05}} remains open:

\begin{question}
\label{q:decaen}
    Must each graph in $\mathscr{G}_3$ have at most three distinct degrees?
\end{question}

Van Dam et al.~\cite{manyvalencies15} demonstrated that the condition of having three distinct eigenvalues in Question~\ref{q:decaen} is critical, showing that there exist connected graphs with four or five distinct eigenvalues and arbitrarily many valencies.

Muzychuk and Klin~\cite{MUZYCHUK1998191} approached the problem of characterising of $\mathscr{G}_3$ graphs by examining the rank of the \emph{Weisfeiler-Leman closure}, which we refer to as the \emph{coherent rank}. 
Strongly regular graphs can be characterised as graphs whose coherent rank is equal to 3. 
Loosely speaking, the coherent rank can be thought of as a measure of how close a graph is to being strongly regular. 
Muzychuk and Klin characterised the graphs in $\mathscr{G}_3$ that have a coherent rank of at most 6 but claimed without proof~\cite[Proposition 6.2]{MUZYCHUK1998191} that no graph in $\mathscr{G}_3$ has a coherent rank of 7 or 8. 
Our first contribution is to prove that no graph in $\mathscr{G}_3$ has a coherent rank of 7. 
However, we show that there do exist graphs in $\mathscr{G}_3$ with a coherent rank of 8, and we characterise such graphs in Theorem~\ref{thm:rank8}, showing that they correspond to certain symmetric $2$-designs.

We continue the characterisation program of Muzychuk and Klin by characterising graphs in $\mathscr{G}_3$ with coherent rank 9. 
One consequence is the discovery of a new biregular graph in $\mathscr{G}_3$ (see Example~\ref{ex:new}). 
In fact, it turns out that graphs in $\mathscr{G}_3$ with coherent rank 9 correspond to certain quasi-symmetric designs (see Section~\ref{sec:9}). 
Quasi-symmetric designs, particularly affine designs~\cite{VANDAM1998101}, have been used to construct an infinite family of graphs in $\mathscr{G}_3$. 
Rowlinson~\cite{rowlinson} asked whether graphs in $\mathscr{G}_3$ can be constructed from quasi-symmetric designs that are not of affine type. 
We answer Rowlinson's question in the affirmative and characterise the quasi-symmetric designs that correspond to graphs in $\mathscr{G}_3$ with coherent rank 9.

Our next contribution is a lower bound for the coherent rank of a graph in $\mathscr{G}_3$ that has three distinct degrees. We show that if a graph in $\mathscr{G}_3$ has three distinct degrees, then its coherent rank is at least 14 (see Theorem~\ref{thm:rank14lb}). This bound has the potential to be sharp contingent on the existence of a certain quasi-symmetric design (see Example~\ref{ex:14}).

All currently known infinite families of graphs in $\mathscr{G}_3$ have a fixed coherent rank. Our final contribution is a conjecturally infinite construction of biregular graphs in $\mathscr{G}_3$ with conjecturally arbitrarily large coherent ranks. Each graph in this family corresponds to a prime power in a certain recurrence sequence and is obtained by switching the block graph of an orthogonal array (see Section~\ref{sec:latin}).

All graphs in this paper are finite, undirected, and simple. The \textbf{spectrum} and \textbf{eigenvalues} of a graph are defined to be those of its adjacency matrix. The paper is organised as follows. 
In Section~\ref{sec:basic}, we develop the foundational theory of graphs with three distinct eigenvalues. In Section~\ref{sec:cr}, we introduce the coherent closure and coherent rank of a graph and present Muzychuk and Klin's characterisation of graphs in $\mathscr{G}_3$ whose coherent rank is small. 
In Section~\ref{sec:8}, we characterise the graphs in $\mathscr{G}_3$ that have coherent rank 8. 
In Section~\ref{sec:9}, we characterise the graphs in $\mathscr{G}_3$ that have coherent rank 9. 
In Section~\ref{sec:tri}, we prove a lower bound for the coherent rank of a triregular graph in $\mathscr{G}_3$. In Section~\ref{sec:latin}, we exhibit a new family of biregular graphs in $\mathscr{G}_3$ obtained from switching the block graphs of orthogonal arrays. Finally, in Section~\ref{sec:open}, we provide a selection of open problems that emerged during our investigations.

\section{Graphs with three distinct eigenvalues}
\label{sec:basic}

In this section, we introduce some basic properties of graphs in $\mathscr G_3$ and tools, which we will use to establish our results, below.

\subsection{Fundamentals}

We denote by $I_n$, $J_n$, $O_n$, and $\mathbf 1_n$, the identity matrix, all-ones matrix, zero matrix, and all-ones (column) vector of order $n$, respectively. 
We merely write $I$, $J$, $O$, and $\mathbf 1$ when the order can be determined from context, or in the case of $J$ and $O$ when the matrix is not square. 
A graph $(V,E)$ is called \textbf{empty} or \textbf{complete} if $E = \emptyset$ or $E = \binom{V}{2}$, respectively.
Denote by $K_v$ the complete graph on $v$ vertices.
Suppose that $\Gamma$ has degree sequence $\{ [k_1]^{n_1}, [k_2]^{n_2}, \dots, [k_d]^{n_d}\}$.
If $d = 1$ then each vertex has $k_1$ neighbours and $\Gamma$ is called \textbf{regular}.
In the cases when $d = 2$ or $d=3$, we call $\Gamma$ \textbf{biregular} or \textbf{triregular}, respectively.

For fixed $\theta_0 > \theta_1 > \theta_2$, define $\mathscr{G}_3(\theta_0,\theta_1,\theta_2) \subset \mathscr{G}_3$ as the subset of connected graphs with eigenvalues $\theta_0$, $\theta_1$, and $\theta_2$. 
Let $\Gamma$ be a graph with $r+1$ distinct eigenvalues $\theta_0 > \theta_1 > \dots > \theta_r$.
The spectrum of $\Gamma$ is denoted by $\operatorname{spec}(\Gamma)$ and is written as $\operatorname{spec}(\Gamma) = \left \{ [\theta_0]^{m_0},[\theta_1]^{m_1}, \dots, [\theta_r]^{m_r} \right \}$, where $m_0, m_1, \dots, m_r$ are the multiplicities of $\theta_0, \theta_1, \dots, \theta_r$ respectively.
By the Perron-Frobenius Theorem~\cite[Theorem 2.2.1]{Brouwer:SpectraGraphs}, a graph $\Gamma \in \mathscr{G}_3(\theta_0,\theta_1,\theta_2)$ has spectrum $\left \{ [\theta_0]^{1},[\theta_1]^{m_1}, [\theta_2]^{m_2} \right \}$, where $m_1$ and $m_2$ are positive integers whose sum is $1$ less than the order of $\Gamma$.
Let $A(\Gamma)$ denote the adjacency matrix of a graph $\Gamma$, and $V(\Gamma)$ its vertex set. 
Again, by the Perron-Frobenius Theorem, if $\Gamma \in \mathscr{G}_3(\theta_0,\theta_1,\theta_2)$, then
\begin{equation}
\label{eqn:3ev}
A(\Gamma)^2 - (\theta_1+\theta_2)A(\Gamma)+\theta_1\theta_2 I = \bm \alpha \bm \alpha ^\top,
\end{equation}
where $\bm \alpha = (\alpha_{\mathsf x})_{\mathsf x \in V(\Gamma)}$ is a positive (column) eigenvector for $\theta_0$. 
Conversely, if a connected graph $\Gamma$ satisfies \eqref{eqn:3ev} for some $\theta_1$, $\theta_2$, and $\bm \alpha$, then $\Gamma \in \mathscr{G}_3$. 
The diagonal entries of \eqref{eqn:3ev} show that the degree $d_{\mathsf x}$ of vertex $\mathsf x$ is $\alpha_{\mathsf x}^2-\theta_1\theta_2$. 
Alternatively, for each $\mathsf x \in V(\Gamma)$, we can write $\alpha_{\mathsf x} = \sqrt{d_{\mathsf x}+\theta_1\theta_2}$.
Moreover, \eqref{eqn:3ev} implies any graph in $\mathscr{G}_3$ must have diameter 2.

The following theorem due to Cauchy~\cite[Theorem 2.5.1]{Brouwer:SpectraGraphs}, will be used repeatedly, throughout.

\begin{theorem}[Eigenvalue Interlacing Theorem]
\label{thm:interlacing}
Let $A$ be a real symmetric $n \times n$ matrix, and $B$ be an $m \times m$ principal submatrix of $A$. 
Suppose $A$ has eigenvalues $\lambda_1 \leqslant \dots \leqslant \lambda_n$, and $B$ has eigenvalues $\mu_1 \leqslant \dots \leqslant \mu_m$.
Then \[
\lambda_k \leqslant \mu_k \leqslant \lambda_{n - m+k}.
\]
\end{theorem}

For a square matrix $M$ and $U$ a subset of its rows, denote by $M[U]$ the principal submatrix of $M$ induced on the rows and columns corresponding to $U$.
Let $\Gamma = (V,E)$ be a graph.
Similarly, for a subset $U \subset V$, denote by $\Gamma[U]$ the subgraph of $\Gamma$ induced on the vertices in $U$.
Thus, $A(\Gamma[U]) = A(\Gamma)[U]$.

Suppose $\Gamma \in \mathscr G_3(\theta_0,\theta_1,\theta_2)$.
Then $\Gamma$ is connected and not complete.
Hence, $\Gamma$ must contain the path graph $P_3 = K_{1,2}$ as an induced subgraph.
Since $P_3 \in \mathscr G_3(\sqrt{2},0,-\sqrt{2})$, by Theorem~\ref{thm:interlacing}, we must have
\begin{equation}
\label{eqn:basicevineq}
    \theta_1 \geqslant 0 \text{ and } \theta_2 \leqslant -\sqrt{2}.
\end{equation}

Regular graphs in $\mathscr{G}_3$ are known as \emph{strongly regular graphs}.
A \textbf{strongly regular graph} $\Gamma$ with parameters $(v,k,a,c)$ is a $k$-regular graph on $v$ vertices such that each pair of adjacent vertices in $V(\Gamma)$ has precisely $a$ common neighbours and each pair of non-adjacent vertices has precisely $c$ common neighbours.
Denote by $\operatorname{SRG}(v,k,a,c)$ the set of all strongly regular graphs that have parameters $(v,k,a,c)$.
If $A = A(\Gamma)$ then the $(\mathsf x,\mathsf y)$-entry of $A^2$ is equal to the number of $2$-walks from $\mathsf x$ to $\mathsf y$.
Thus, if $\Gamma \in \operatorname{SRG}(v,k,a,c)$ then
\[
A^2 = kI + a A + c(J-I-A).
\]
Furthermore, if $\Gamma$ is connected then 
\[
\Gamma \in \mathscr G_3\left (k,\frac{a-c+\sqrt{(a-c)^2+4(k-c)}}{2},\frac{a-c-\sqrt{(a-c)^2+4(k-c)}}{2} \right ).
\]
Conversely, it also follows that if $\Gamma \in \mathscr G_3(\theta_0,\theta_1,\theta_2)$ is regular with $v$ vertices then 
\[\Gamma \in \operatorname{SRG}(v,\theta_0,\theta_0-1+(\theta_1+1)(\theta_2+1),\theta_0+\theta_1\theta_2).
\]

\subsection{Cones and bipartite graphs}

A complete bipartite graph $K_{a,b}$ belongs to the set $\mathscr G_3(-\sqrt{ab},0,\sqrt{ab})$.
Moreover, $K_{a,b}$ has spectrum $\left \{ [-\sqrt{ab}]^1, [0]^{a+b-2}, [\sqrt{ab}]^1 \right \}$.
The \textbf{cone over} a graph $\Gamma$ is formed by adjoining a vertex adjacent to every vertex of $\Gamma$.
Muzychuk and Klin~\cite{MUZYCHUK1998191} showed that if $\Gamma$ is a $v$-vertex, $k$-regular strongly regular graph with smallest eigenvalue $\theta_2$ then the cone over $\Gamma$ is in $\mathscr G_3$ if and only if $v = \theta_2(\theta_2-k)$.

The \textbf{complement} of a graph $\Gamma = (V,E)$ is defined to be the graph $(V,\binom{V}{2}-E)$ and is denoted by $\overline{\Gamma}$.
The following theorem is a useful characterisation of cones and complete bipartite graphs.

\begin{theorem}[{\cite[Theorem 2.2]{Cheng_2016}}]
\label{thm:cheng}
    Let $\Gamma \in \mathscr G_3$ be a non-regular graph whose complement is disconnected.
    Then $\Gamma$ is a cone or complete bipartite.
\end{theorem}

Next, Van Dam~\cite[Section 4.2]{VANDAM1998101} answered Question~\ref{q:decaen} in the affirmative for cones.

\begin{lemma}[{\cite[Section 4.2]{VANDAM1998101}}]
    \label{lem:VDcone}
    Let $\Gamma$ be a cone in $\mathscr G_3$.
    Then either $\Gamma$ is a cone over a strongly regular graph or $\Gamma$ is triregular.
    Furthermore, if $\Gamma \in \mathscr G_3(\theta_0,\theta_1,\theta_2)$ is triregular with distinct degrees $k_1 > k_2 > k_3$ then $\sqrt{k_1+\theta_1\theta_2} = \sqrt{k_2+\theta_1\theta_2}+\sqrt{k_3+\theta_1\theta_2}$.
\end{lemma}

Two examples of triregular cones in $\mathscr G_3$ were discovered by Bridges and Mena~\cite{BM81}.
See Table~\ref{tab:3distinctValencies}, below.
In fact, Bridges and Mena~\cite{BM81} classified the triregular cones in $\mathscr G_3$ whose smallest two eigenvalues $\theta_1$ and $\theta_2$ satisfy $\theta_1+\theta_2 = 0$.
There are potentially three such graphs, with the existence of the third possibility still an open problem (see \cite[Section 9]{Cheng_2016}).

Van Dam~{\cite[Section 4.2]{VANDAM1998101}} showed that biregular cones in $\Gamma \in \mathscr G_3$ are precisely cones over strongly regular graphs.
We record this result as the following lemma.

\begin{lemma}
\label{lem:higmanType7}
    Let $\Gamma \in \mathscr G_3$ be a biregular cone.
    Then $\Gamma$ is a cone over a strongly regular graph.
\end{lemma}

Non-regular complete bipartite graphs can be characterised as the non-regular graphs in $\mathscr G_3$ with second largest eigenvalue $0$.

\begin{theorem}[{\cite[Corollary 2.3]{Cheng_2016}}]
\label{thm:chengbipartite}
    Let $\Gamma \in \mathscr G_3$ be a non-regular graph.
    Then the following are equivalent.
    \begin{itemize}
        \item[(i)] $\Gamma$ is bipartite;
        \item[(ii)] $\Gamma$ is complete bipartite;
        \item[(iii)] $\theta_1 = 0$.
    \end{itemize}
\end{theorem}

\subsection{Valencies and equitable partitions}

Let $\Gamma= (V,E)$ be a graph and let $\pi = \{\pi_1,\dots,\pi_t\}$ be a partition of the vertex set $V$.
We call $\pi$ \textbf{equitable} if, for each $i,j \in \{1,\dots,t\}$ there exists $n_{ij}$ such that each vertex in $\pi_i$ has $n_{ij}$ neighbours in $\pi_j$.
Denote by $Q_\pi(\Gamma)$ the \textbf{quotient matrix} $(n_{ij})_{i,j \in \{1,\dots,t\}}$ that corresponds to the vertex partition $\pi$.
\begin{lemma}[{\cite[Lemma 2.3.1]{Brouwer:SpectraGraphs}}]
\label{lem:equitable}
    Let $\Gamma$ be a graph with equitable partition $\pi$ of $V(\Gamma)$.
    Then $\operatorname{spec}(Q_\pi(\Gamma)) \subset \operatorname{spec}(\Gamma)$.
\end{lemma}

Suppose $\Gamma$ has degree sequence $\{ [k_1]^{n_1}, [k_2]^{n_2}, \dots, [k_d]^{n_d}\}$.
For each $i \in \{1,\dots,d\}$, let $V_i$ denote the subset of vertices whose degree is $k_i$.
Clearly, $\{V_1,\dots,V_d\}$ is a partition of $V(\Gamma)$.
We call this partition the \textbf{valency partition} of $\Gamma$ and we denote it by $\mathfrak D(\Gamma)$.

\begin{theorem}[{\cite[Section 4]{VANDAM1998101}}]
    \label{thm:equitableValency}
    Let $\Gamma \in \mathscr G_3(\theta_0,\theta_1,\theta_2)$ have at most three distinct valencies.
    Then the valency partition is equitable.
    Furthermore, $\theta_0$ is an eigenvalue of $Q_{\mathfrak D(\Gamma)}(\Gamma)$.
\end{theorem}

Biregular, non-bipartite graphs in $\mathscr G_3$ are subject to stronger conditions, which we list in Theorem~\ref{thm:bireg}.

\begin{theorem}[{\cite[Theorem 4.3]{Cheng_2016}}]
    \label{thm:bireg}
    Let $\Gamma$ be a non-bipartite biregular graph in $\mathscr G_3(\theta_0,\theta_1,\theta_2)$ with degree sequence $\{ [k_1]^{n_1}, [k_2]^{n_2}\}$ and $\mathfrak D(\Gamma) = \{V_1,V_2\}$. 
	Then the following conditions hold:
	\begin{enumerate}
		\item[(i)] All eigenvalues of $\Gamma$ are integers.
        \item[(ii)] The quotient matrix $Q_{\mathfrak D(\Gamma)}(\Gamma)$ has eigenvalues $\theta_0$ and $\theta$, where $\theta \in \{\theta_1,\theta_2\}$.
		\item[(iii)] $\sqrt{(k_1+\theta_1\theta_2)(k_2+\theta_1\theta_2)} = -\theta(\theta' +1)$ where $\{ \theta, \theta'\} = \{ \theta_1, \theta_2\}$.
            \item[(iv)] If $\Gamma[V_1]$ or $\Gamma[V_2]$ is empty then $\theta = \theta_2$.
  \end{enumerate}
\end{theorem}

\section{Coherent configurations and algebras}
\label{sec:cr}
In this section, we define the coherent rank of a graph.

\subsection{The coherent rank}
Let $\mathfrak X$ be a finite set and let $\mathfrak R = \{R_1,\dots,R_{\mathnormal r}\}$ be a set of binary relations on $\mathfrak X$.
For each $R_i$ let $A_i \in \operatorname{Mat}_{\mathfrak X}(\{0,1\})$ be defined such that its $(\mathsf x,\mathsf y)$ entry is $1$ if $(\mathsf x,\mathsf y) \in R_i$ and $0$ otherwise.
Suppose that 
\begin{enumerate}
    \item[(CC1)] $\displaystyle \sum_{i=1}^{\mathnormal r} A_i = J$;
    \item[(CC2)] For each $i \in \{1,\dots,{\mathnormal r}\}$ there exists $j \in \{1,\dots,{\mathnormal r}\}$ such that $A_i^\transpose = A_j$;
    \item[(CC3)] There exists a subset $\Delta \subset \{1,\dots,{\mathnormal r}\}$ such that $\sum_{i \in \Delta} A_i = I$;
    \item[(CC4)] $\displaystyle A_iA_j = \sum_{k=1}^{\mathnormal r} p_{i,j}^k A_k$, for each $i,j \in \{1,\dots,{\mathnormal r}\}$.
\end{enumerate}
Then $({\mathfrak X},{\mathfrak R})$ is called a \textbf{coherent configuration} of \textbf{rank} ${\mathnormal r} = |{\mathfrak R}|$.
The set ${\mathfrak X}$ is called the \textbf{point-set} of the coherent configuration.

For each $i \in \Delta$, we call the subset ${\mathfrak X}_i := \{ \mathsf x \in {\mathfrak X} \; : \; (\mathsf x,\mathsf x) \in R_i\}$ a \textbf{fibre} of the coherent configuration.
Clearly, the fibres form a partition of the point-set $\mathfrak X$.
When $|\Delta| = 1$, the coherent configuration $({\mathfrak X},{\mathfrak R})$ is called an \textbf{association scheme}.
It follows from (CC4) that, for each $k \in \{1,\dots,{\mathnormal r}\}$, there exists $i$ and $j$ such that $R_k \subset {\mathfrak X}_i \times {\mathfrak X}_j$.
Thus, each subset $\Delta^\prime$ of $\Delta$ induces a coherent configuration with point-set $\bigcup_{i \in \Delta^\prime} {\mathfrak X}_i$.
The \textbf{type} of $({\mathfrak X},{\mathfrak R})$ is defined to be the matrix in $\operatorname{Mat}_\Delta(\mathbb N)$ whose $(i,j)$-entry $t_{ij}$ is equal to the cardinality $|\{ k \; : \; R_k \subset {\mathfrak X}_i \times {\mathfrak X}_j \}|$.
Note that the sum of the entries of the type matrix is equal to ${\mathnormal r}$.
Furthermore, since the type matrix must be symmetric, we omit the entries below the diagonal.
Higman~\cite{Higman1987} established the following restriction on the type matrix.
\begin{lemma}
\label{lem:HigmanType}
For each $i,j \in \Delta$,
    if $t_{ii} \leqslant 5$ and $t_{jj} \leqslant 5$ then $t_{ij} \leqslant \min(t_{ii},t_{jj})$.
\end{lemma}

A \textbf{coherent algebra} is a matrix algebra $\mathcal A \subset \operatorname{Mat}_{\mathfrak X}(\mathbb C)$ that satisfies the following axioms.
\begin{itemize}
    \item[(A1)] $I, J \in \mathcal A$;
    \item[(A2)] $M^\transpose \in \mathcal A$ for each $M \in \mathcal A$;
    \item[(A3)] $MN \in \mathcal A$ and $M \circ N \in \mathcal A$ for each $M,N \in \mathcal A$, where $\circ$ denotes the entrywise product.
\end{itemize}
Each coherent algebra $\mathcal A$ has a unique basis of $\{0,1\}$-matrices $\{A_1,\dots,A_{\mathnormal r}\}$ that corresponds to a coherent configuration $({\mathfrak X}_{\mathcal A},{\mathfrak R}_{\mathcal A})$.
We denote by $\mathfrak F_{\mathcal A}$ the set of fibres of the coherent configuration $({\mathfrak X}_{\mathcal A},{\mathfrak R}_{\mathcal A})$ and we define \textbf{type} of $\mathcal A$ to be that of $({\mathfrak X}_{\mathcal A},{\mathfrak R}_{\mathcal A})$.
Clearly, the intersection of any two coherent algebras is itself a coherent algebra.
We can thus define the \textbf{coherent closure} $\mathcal {W}(\Gamma)$ of $\Gamma$ to be the minimal coherent algebra that contains the adjacency matrix $A(\Gamma)$ of $\Gamma$.
We write $\mathcal {W}(\Gamma) = \langle A_1,\dots,A_{\mathnormal r} \rangle$, where $\{A_1,\dots,A_{\mathnormal r}\}$ is the unique basis of $\{0,1\}$-matrices for $\mathcal {W}(\Gamma)$.

We summarise in the following corollary some straightforward consequences of the above definitions.

\begin{corollary}
\label{cor:allFib}
    Let $\Gamma$ be a graph with coherent closure $\mathcal {W}(\Gamma) = \langle A_1,\dots,A_{\mathnormal r}\rangle$.
    Then 
    \begin{itemize}
        \item[(i)] For each $\mathfrak f \in \mathfrak F_{\mathcal {W}(\Gamma)}$, the matrix $A(\Gamma)[\mathfrak f]$ belongs to the coherent algebra $\langle A_1[\mathfrak f],\dots,A_{\mathnormal r}[\mathfrak f]\rangle$.
        \item[(ii)] If $\mathsf x$ and $\mathsf y$ are in the same fibre of $\mathcal {W}(\Gamma)$ then $\mathsf x$ and $\mathsf y$ have the same degree.
        \item[(iii)] $|\mathfrak D(\Gamma)| \leqslant |\mathfrak F_{\mathcal {W}(\Gamma)}|$.
    Furthermore, in the case of equality, we have $\mathfrak D(\Gamma) = \mathfrak F_{\mathcal {W}(\Gamma)}$.
    \end{itemize}
\end{corollary}

Define the \textbf{coherent rank} of $\Gamma$ to be the rank of its coherent closure $\mathcal {W}(\Gamma)$.
The coherent rank of a complete or empty graph on at least $2$ vertices is $2$ and the coherent rank of a strongly regular graph is $3$.
Conversely, it is straightforward to verify the following proposition.

\begin{proposition}
\label{pro:smallRank}
    Let $\mathcal A$ be a coherent algebra of rank ${\mathnormal r}$.
    \begin{itemize}
        \item[(i)] If ${\mathnormal r}=1$ then $\mathcal A = \langle I_1 \rangle$.
        \item[(ii)] If ${\mathnormal r} = 2$ then $\mathcal A = \langle I, J-I \rangle$.
        \item[(iii)] If ${\mathnormal r}=3$ and $\mathcal A$ contains the adjacency matrix of a graph that is neither empty nor complete then $\mathcal A = \langle I, A, J-I-A \rangle$,
    where $A$ is the adjacency matrix of a strongly regular graph.
    \end{itemize}
\end{proposition}

\subsection{Small coherent rank}

Muzychuk and Klin~\cite{MUZYCHUK1998191} initiated the study of graphs with three eigenvalues that have small coherent rank.
We summarise their findings in the following theorem.

\begin{theorem}[{\cite[Section 6]{MUZYCHUK1998191}}]
\label{thm:muzychuk}
    Let $\Gamma \in \mathscr G_3$.
    Then the coherent rank of $\Gamma$ is
    \begin{itemize}
        \item[(i)] 3 if and only if $\Gamma$ is a strongly regular graph;
        \item[(ii)] 5 if and only if $\Gamma \cong K_{1,\mathsf b}$ where $\mathsf b \geqslant 2$;
        \item[(iii)] 6 if and only if
        \begin{itemize}
            \item $\Gamma \cong K_{\mathsf a,\mathsf b}$ where $2 \leqslant \mathsf a < \mathsf b$; or
            \item $\Gamma$ is a cone over a strongly regular graph.
        \end{itemize}
    \end{itemize}
    Furthermore, the coherent rank of $\Gamma$ cannot be $4$.
\end{theorem}

\begin{remark}
    Note that if we allow disconnected graphs then coherent rank $4$ is possible.
    For example, take the disjoint union of two copies of a strongly regular graph.
\end{remark}

In each case of Theorem~\ref{thm:muzychuk} the coherent closure of $\Gamma$ has type  $\left [ \begin{smallmatrix}3 \end{smallmatrix}\right ]$, $\left [ \begin{smallmatrix}1 & 1 \\  & 2\end{smallmatrix}\right ]$, or $\left [ \begin{smallmatrix}2 & 1 \\  & 2\end{smallmatrix}\right ]$, and $\left [ \begin{smallmatrix}1 & 1 \\  & 3\end{smallmatrix}\right ]$, respectively.
Muzychuk and Klin~\cite[Proposition 6.2]{MUZYCHUK1998191} further claimed without proof that a graph with three distinct eigenvalues cannot have coherent rank 7 or 8.
Contrary to their claim, we will show that coherent rank 8 is in fact possible.
First, we give a proof that a graph with three distinct eigenvalues cannot have coherent rank 7.

\begin{theorem}
\label{thm:wl7}
    Let $\Gamma \in \mathscr G_3$.
    Then the coherent rank of $\Gamma$ is not $7$.
\end{theorem}
\begin{proof}
    Suppose for a contradiction that $\Gamma\in \mathscr G_3$ be a graph with coherent rank $7$ and let $A$ be its adjacency matrix.
    By Theorem~\ref{thm:muzychuk}, the graph $\Gamma$ is not regular.
Using Lemma~\ref{lem:HigmanType} and Corollary~\ref{cor:allFib} (iii), we find that the coherent closure $\mathcal {W}(\Gamma)$ has type  $\left [ \begin{smallmatrix}1 & 1 \\  & 4\end{smallmatrix}\right ]$ or $\left [ \begin{smallmatrix}2 & 1 \\  & 3\end{smallmatrix}\right ]$.
By Corollary~\ref{cor:allFib} (iii), the vertices of $\Gamma$ have two distinct degrees $k_1$ and $k_2$ (say).
Furthermore, since $A \in \mathcal {W}(\Gamma)$ and $\Gamma$ is connected, it follows that each vertex of degree $k_1$ must be adjacent to every vertex of degree $k_2$.
Whence, the complement of $\Gamma$ is disconnected.
Thus, by Theorem~\ref{thm:cheng} and Lemma~\ref{lem:higmanType7} the graph $\Gamma$ is a cone over a strongly regular graph.
But, by Theorem~\ref{thm:muzychuk}, such a graph has coherent rank $6$, which is a contradiction.
\end{proof}

We conclude this section with a corollary of Theorem~\ref{thm:muzychuk} and Theorem~\ref{thm:chengbipartite}.

\begin{corollary}
\label{cor:gt0}
    Let $\Gamma \in \mathscr G_3(\theta_0,0,\theta_2)$.
    Then the coherent rank of $\Gamma$ is at most $6$.
\end{corollary}

\section{Coherent rank 8}
\label{sec:8}

In this section, we characterise the graphs $\Gamma \in \mathscr G_3$ that have coherent rank $8$.

\subsection{Symmetric 2-designs}

A \textbf{design} is a pair $(\mathcal P, \mathcal B)$ consisting of \textbf{point-set} $\mathcal P$ with cardinality $|\mathcal P| = v$ and $\mathcal B$ a family of $k$-sets from $\binom{\mathcal P}{k}$ for some $k \in \{1,\dots,v\}$.\footnote{We allow the case of repeated blocks, although we will show that only designs without repeated blocks are pertinent in our study. (See Corollary~\ref{cor:repeatedblocks}).}
Accordingly, the elements of $\mathcal P$ and $\mathcal B$ are called \textbf{points} and \textbf{blocks} respectively.
If there exists $t$ and $\lambda$ such that each set in $\binom{\mathcal P}{t}$ is contained in precisely $\lambda$ blocks then $(\mathcal P, \mathcal B)$ is called a $t$-$(v,k,\lambda)$ \textbf{design}.
We may write $t$-design in place of $t$-$(v,k,\lambda)$ design if the parameters are not required to be specified.
Suppose that every pair of blocks intersects in $x$ points.
Then the design $(\mathcal P,\mathcal B)$ is called a \textbf{symmetric} design.

Let $(\mathcal P, \mathcal B)$ be a $2$-$(v,k,\lambda)$ design.
Write $b = |\mathcal B|$ for the number of blocks and $r$ for the number of blocks that contain a particular point.
Standard double-counting arguments yield the following equations.
\begin{align}
    vr &= bk; \label{eqn:2des1} \\
    \lambda(v-1) &= r(k-1) \label{eqn:2des2}.
\end{align}
We also have the fundamental inequality due to Fisher~\cite{fisher}:
\begin{align}
    b &\geqslant v. \label{eqn:fisher}
\end{align}
The \textbf{incidence matrix} of $(\mathcal P, \mathcal B)$ is a $v \times b$ matrix whose $(i,j)$-entry is equal to $1$ if $i \in j$ and $0$ otherwise.
Let $M$ be the incidence matrix for $(\mathcal P, \mathcal B)$.
Then $M$ satisfies $MJ = rJ$, $M^\transpose J = kJ$, and 
\begin{align}
   MM^\transpose &= rI+\lambda(J-I). \label{eqn:2desmmt} 
\end{align}
Furthermore, if $(\mathcal P, \mathcal B)$ is a symmetric design then
\begin{equation}
\label{eqn:mtmsym}
    M^\transpose M = kI + x(J-I)
\end{equation}
Note that symmetric designs correspond to equality in Fisher's inequality \eqref{eqn:fisher}.
Moreover, we obtain
\begin{align}
    \lambda(v-1) &= k(k-1) \label{eqn:2desSYM}.
\end{align}
Now we state a tool that we will require below.

\begin{lemma}[{\cite[Lemma 2.9.2]{Brouwer:SpectraGraphs}}]
\label{lem:LGtoL}
	Let $M$ be an $n \times m$ real matrix.
	Then $M^\transpose M$ and $MM^\transpose$ have the same nonzero eigenvalues
	(including multiplicities).
\end{lemma}

\subsection{Graphs with three eigenvalues and coherent rank 8}

Next, we show that graphs in $\mathscr G_3$ whose coherent closure has type $\left [ \begin{smallmatrix}2 & 2 \\  & 2 \end{smallmatrix}\right ]$ correspond to certain symmetric designs.

\begin{proposition}
\label{pro:rank8}
Let $\Gamma \in \mathscr G_3$ such that $\mathcal {W}(\Gamma)$ has type $\left [ \begin{smallmatrix}2 & 2 \\  & 2 \end{smallmatrix}\right ]$.
Then the adjacency matrix of $\Gamma$ has the form $\left [ \begin{smallmatrix} O & M \\ M^\transpose & J-I \end{smallmatrix}\right ]$ where $M$ is the incidence matrix of a symmetric $2$-$(\lambda^3-\lambda+1,\lambda^2,\lambda)$ design for some $\lambda \in \mathbb N$.
\end{proposition}
\begin{proof}
Suppose $A=A(\Gamma)$ is the adjacency matrix of $\Gamma$.
By Theorem~\ref{thm:muzychuk} and Corollary~\ref{cor:allFib} (iii), the graph $\Gamma$ must be biregular with degrees $k_1$ and $k_2$ (say).
    Suppose $\mathfrak D(\Gamma) = \{V_1,V_2\}$.
    By Corollary~\ref{cor:allFib} (iii), we can write $A = \left [\begin{smallmatrix}
    A(\Gamma[V_1]) & M \\
    M^\transpose & A(\Gamma[V_2])
\end{smallmatrix}\right ]$ and
    \[
    A^2 =\begin{bmatrix}
    A(\Gamma[V_1])^2+MM^\transpose & A(\Gamma[V_1]) M+MA(\Gamma[V_2]) \\
    M^\transpose A(\Gamma[V_1])+A(\Gamma[V_2])M^\transpose & M^\transpose M +A(\Gamma[V_2])^2
\end{bmatrix}.
\]
Apply \eqref{eqn:3ev} to obtain
\begin{equation}
\label{eqn:biregGen2222}
A^2
-(\theta_1+\theta_2)A
+\theta_1\theta_2 I = 
\begin{bmatrix}
    (k_1+\theta_1\theta_2) J & \sqrt{(k_1+\theta_1\theta_2)(k_2+\theta_1\theta_2)} J \\
    \sqrt{(k_1+\theta_1\theta_2)(k_2+\theta_1\theta_2)} J & (k_2+\theta_1\theta_2) J
\end{bmatrix}.
\end{equation}
By Corollary~\ref{cor:allFib} (i) and Proposition~\ref{pro:smallRank}, the subgraph $\Gamma[V_1] = K_{n_1}$ or $\overline K_{n_1}$, and the subgraph $\Gamma[V_2] = K_{n_2}$ or $\overline K_{n_2}$, that is, the matrix $A(\Gamma[V_1]) = \varepsilon_{11}(J-I)$ and the matrix $A(\Gamma[V_2]) = \varepsilon_{22}(J-I)$ for some $\varepsilon_{11}, \varepsilon_{22} \in \{0,1\}$.
From the top left block of \eqref{eqn:biregGen2222}, we obtain
\begin{equation}
\label{eqn:bbt2222}
    MM^\transpose = \begin{cases}
 (k_{12}+\theta_1\theta_2)J_{n_1}-\theta_1\theta_2I, & \text{ if $\Gamma[V_1] = \overline K_{n_1}$};  \\
(k_{12}+(\theta_1+1)(\theta_2+1))J_{n_1}-(\theta_1+1)(\theta_2+1)I, & \text{ if $\Gamma[V_1] =  K_{n_1}$}.
\end{cases}
\end{equation}
Similar to the above, we can use the bottom-right block of \eqref{eqn:biregGen2222} to deduce that 
\begin{equation}
\label{eqn:btb2222}
    M^\transpose M = \begin{cases}
 (k_{21}+\theta_1\theta_2)J_{n_2}-\theta_1\theta_2I, & \text{ if $\Gamma[V_2] = \overline K_{n_2}$};  \\
(k_{21}+(\theta_1+1)(\theta_2+1))J_{n_2}-(\theta_1+1)(\theta_2+1)I, & \text{ if $\Gamma[V_2] =  K_{n_2}$}.
\end{cases}
\end{equation}

Next, we show that, without loss of generality, we can assume that $\Gamma[V_1]$ is empty and $\Gamma[V_2]$ is complete.
By Theorem~\ref{thm:chengbipartite}, $\Gamma[V_1]$ and $\Gamma[V_2]$ cannot both be empty.
We further claim that $\Gamma[V_1]$ and $\Gamma[V_2]$ cannot both be complete.
Suppose (for a contradiction) that both $\Gamma[V_1]$ and $\Gamma[V_2]$ are complete.
Using \eqref{eqn:bbt2222} and \eqref{eqn:btb2222} together with Lemma~\ref{lem:LGtoL}, we find that $n_1 = n_2$ and $k_{12} = k_{21}$.
This implies that $\Gamma$ is regular, which gives a contradiction.

Hence, we may assume that $\Gamma[V_1]$ is empty and $\Gamma[V_2]$ is complete.
Apply Lemma~\ref{lem:LGtoL} to find that $\theta_1\theta_2 = (\theta_1+1)(\theta_2+1)$.
This implies that $\theta_1 + \theta_2 = -1$.
Furthermore, $M$ must be the incidence matrix of a symmetric $2$-design $\mathcal D$.
Suppose $\mathcal D$ has parameters $(v,k,\lambda)$.
Then $v = b = n_1 = n_2$ and $k = r = \lambda(v-1)/(k-1)$.

Now, we can write
\[
A^2 =  \begin{bmatrix}
(k-\lambda)I + \lambda J & kJ - M \\ kJ - M^\transpose & (k-\lambda+1) I + (v+\lambda-2)J 
\end{bmatrix}.
\]
Observe that the off-diagonal entries and the diagonal entries of the diagonal blocks of \eqref{eqn:biregGen2222} are equal.
Equating the diagonal and off-diagonal entries of the top-left block yields $\theta_1\theta_2 = \lambda - k$.
Lastly, we use the equation
\begin{equation}
\label{eqn:3ev3}
    (A-\theta_0 I)(A-\theta_1 I)(A-\theta_2 I) = O.
\end{equation}
The top-left block of $A^3$ is equal to $(k^2-\lambda)J + (\lambda - k)I$.
Thus, the top-left block of \eqref{eqn:3ev3} is
\[
(k^2-\lambda)J + (\lambda - k)I - (\theta_0+\theta_1+\theta_2)((k-\lambda)I + \lambda J)-\theta_0\theta_1\theta_2I=O.
\]
Using $\theta_1+\theta_2=-1$ together with the off-diagonal entries of the above yields $\theta_0 = k^2/\lambda$.
The top-right block of $A^3$ is equal to $(k-\lambda + 1) M + k(\lambda + v - 2) J$.
Thus, the top-right block of \eqref{eqn:3ev3} is
\[
(k-\lambda + 1) M + k(\lambda + v - 2) J-(\theta_0+\theta_1+\theta_2)(kJ-M)+(\theta_0\theta_1+\theta_0\theta_2+\theta_1\theta_2)M=O.
\]
Combining the coefficients of $J$ yields $k(\lambda+v-2)=k(\theta_0-1)$, which simplifies to the equation $\theta_0 = v+\lambda-1$.
The two expressions for $\theta_0$ combined with the equation $\lambda(v-1)=k(k-1)$ yields $k = \lambda^2$.
Furthermore, $v = \lambda^3-\lambda+1$.
\end{proof}

We can strengthen the above result to show that if a graph in $\mathscr G_3$ has coherent rank 8 then its coherent closure has type $\left [ \begin{smallmatrix}2 & 2 \\  & 2 \end{smallmatrix}\right ]$.

\begin{theorem}
\label{thm:rank8}
    Let $\Gamma \in \mathscr G_3$ with coherent rank $8$.
    Then the adjacency matrix of $\Gamma$ has the form $\left [ \begin{smallmatrix} O & M \\ M^\transpose & J-I \end{smallmatrix}\right ]$ where $M$ is the incidence matrix of a symmetric $2$-$(\lambda^3-\lambda+1,\lambda^2,\lambda)$ design for some $\lambda \in \mathbb N$.
\end{theorem}
\begin{proof}
By Theorem~\ref{thm:muzychuk}, the graph $\Gamma$ is not regular.
If $\mathcal {W}(\Gamma)$ has type $\left [ \begin{smallmatrix}2 & 2 \\  & 2 \end{smallmatrix}\right ]$ then the conclusion follows from Proposition~\ref{pro:rank8}.
    Otherwise, by Lemma~\ref{lem:HigmanType} and Corollary~\ref{cor:allFib} (iii), the coherent closure $\mathcal {W}(\Gamma)$ has type $\left [ \begin{smallmatrix}1 & 1 \\  & 5 \end{smallmatrix}\right ]$, $\left [ \begin{smallmatrix}2 & 1 \\  & 4 \end{smallmatrix}\right ]$, or $\left [ \begin{smallmatrix}3 & 1 \\  & 3 \end{smallmatrix}\right ]$.
    In each case, the complement of $\Gamma$ is disconnected.
    By Theorem~\ref{thm:cheng} together with Lemma~\ref{lem:higmanType7} and Theorem~\ref{thm:muzychuk}, such a graph has coherent rank at most $6$, a contradiction.
\end{proof}

\begin{example}
    \label{ex:rank8}
    Van Dam~\cite[Section 2.3]{VANDAM1998101} provided an infinite family of graphs corresponding to Theorem~\ref{thm:rank8}.
\end{example}

The proof of the next proposition can be obtained by the same techniques used in the proof of Proposition~\ref{pro:rank8}.

\begin{proposition}
    \label{pro:22}
    Let $\Gamma \in \mathscr G_3$ such that $A(\Gamma) = \left [ \begin{smallmatrix} \varepsilon_{1}(J-I) & M \\ M^\transpose & \varepsilon_{2}(J-I) \end{smallmatrix}\right ]$ for some matrix $M$ with $M\mathbf 1 = r\mathbf 1$, $M^\transpose \mathbf 1 = k\mathbf 1$, and $\varepsilon_{1}, \varepsilon_{2} \in \{0,1\}$.
    Then the coherent rank of $\Gamma$ is at most $8$.
\end{proposition}
\begin{proof}
Suppose that $M$ is an $n_1 \times n_2$ matrix.
Clearly, $\Gamma$ has at most two distinct degrees.
If $\Gamma$ is regular then $\Gamma$ has coherent rank $3$.
Otherwise, we can assume that $\Gamma$ is biregular with degrees $k_1$ and $k_2$ (say).
The top-left block of \eqref{eqn:3ev} yields 
$$MM^\transpose = (k_1+\theta_1\theta_2-\varepsilon_1(n_1-2-\theta_1-\theta_2))J-(\varepsilon_1(1+\theta_1+\theta_2)+\theta_1\theta_2)I.$$
The bottom-right block of \eqref{eqn:3ev} yields 
$$M^\transpose M = (k_2+\theta_1\theta_2-\varepsilon_2(n_2-2-\theta_1-\theta_2))J-(\varepsilon_2(1+\theta_1+\theta_2)+\theta_1\theta_2)I.$$
Hence, the coherent closure $\mathcal {W}(\Gamma)$ must be a subalgebra of the coherent algebra
    \[
    \langle \left [\begin{smallmatrix}
        I & O \\ O & O
    \end{smallmatrix} \right ], \left [\begin{smallmatrix}
        J-I & O \\ O & O
    \end{smallmatrix} \right ], \left [\begin{smallmatrix}
        O & M \\ O & O
    \end{smallmatrix} \right ], \left [\begin{smallmatrix}
        O & J-M \\ O & O
    \end{smallmatrix} \right ], \left [\begin{smallmatrix}
        O & O \\ M^\transpose & O
    \end{smallmatrix} \right ], \left [\begin{smallmatrix}
        O & O \\ J-M^\transpose & O
    \end{smallmatrix} \right ],
    \left [\begin{smallmatrix}
        O & O \\ O & I
    \end{smallmatrix} \right ], \left [\begin{smallmatrix}
        O & O \\ O & J-I
    \end{smallmatrix} \right ]\rangle,
    \]
    which has rank at most $8$.
\end{proof}

\section{Coherent rank 9}
\label{sec:9}

In this section, we characterise the graphs $\Gamma \in \mathscr G_3$ that have coherent rank $9$.

\subsection{The total graph and whole graph of a QSD}

Let $(\mathcal P, \mathcal B)$ be a $2$-$(v,k,\lambda)$ design.
Suppose that every pair of blocks in $\mathcal B$ intersects in either $x$ or $y$ points and both occur.
If $x \ne y$ then $(\mathcal P,\mathcal B)$ is called a \textbf{quasi-symmetric} design.
The numbers $x$ and $y$ are called \textbf{intersection numbers}.
Let $M$ be the incidence matrix of a quasi-symmetric $2$-$(v,k,\lambda)$ design $\mathcal Q$ with intersection numbers $x$ and $y$.
Then
\begin{equation}
\label{eqn:mtmquasisym}
    M^\transpose M = kI + xX + y(J-I-X),
\end{equation}
where $X$ is the $\{0,1\}$-matrix indexed by $\mathcal B$ whose $(i,j)$-entry is $1$ precisely when $|i \cap j| = x$.

The graph whose adjacency matrix is $X$ is called the $x$-\textbf{block graph} of $\mathcal Q$ and is denoted by $\mathsf B_x(\mathcal Q)$.
Note that $\mathsf B_y(\mathcal Q) = \overline{\mathsf B_x(\mathcal Q)}$.
For the sake of brevity, we refer to $\mathcal Q$ as a \textbf{QSD} with parameters $(v,k,\lambda;b,r,\{x,y\})$; recall that $r$ and $b$ can be recovered from $v$, $k$, and $\lambda$ via \eqref{eqn:2des1} and \eqref{eqn:2des2}.
The $x$\textbf{-total graph} of $\mathcal Q$, which we denote by $\mathsf T_x(\mathcal Q)$, is defined to be the graph with adjacency matrix
\[
\begin{bmatrix}
    O & M \\
    M^\transpose & A(\mathsf B_x(\mathcal Q))
\end{bmatrix}.
\]
We define the $x$\textbf{-whole graph} of $\mathcal Q$, which we denote by $\mathsf W_x(\mathcal Q)$, to be the graph with adjacency matrix
\[
\begin{bmatrix}
    J-I & M \\
    M^\transpose & A(\mathsf B_x(\mathcal Q))
\end{bmatrix}.
\]

The block graph of a QSD $\mathcal Q$ is a strongly regular graph~\cite[Theorem 48.10]{shrikhandeHCD} and its eigenvalues can be expressed in terms of the parameters of $\mathcal Q$:

\begin{theorem}
\label{thm:blockspec}
The $x$-block graph $\mathsf B_x(\mathcal Q)$ of a QSD with parameters $(v,k,\lambda;b,r,\{x,y\})$ has spectrum
        $$\operatorname{spec}\left (\mathsf B_x(\mathcal Q) \right ) = \left \{ \left [\frac{k(r-1) - y(b-1)}{x-y} \right ]^1, \left [\frac{r-\lambda-k + y}{x-y}\right ]^{v-1}, \left [\frac{y-k}{x-y} \right ]^{b-v} \right\}.$$
\end{theorem}

To avoid potential confusion from Theorem~\ref{thm:blockspec}, we make a clarifying remark.

\begin{remark}
    \label{rem:blockGraph3ev}
    Note that Theorem~\ref{thm:blockspec} does not claim that the $x$-block graph of a QSD is necessarily connected or that it has precisely three distinct eigenvalues.
    For example, it is not necessary that 
    \[
    \frac{k(r-1) - y(b-1)}{x-y} \ne \frac{r-\lambda-k + y}{x-y}.
    \]
Indeed, one can obtain a QSD by pooling together $m \geqslant 2$ copies of the blocks of a symmetric $2$-$(v,k,\lambda)$ design.
The resulting QSD $\mathcal Q$ has parameters $(v,k,m\lambda; mv, mk, \{k,\lambda\})$.
Its block graph $\mathsf B_k(\mathcal Q)$ is the disjoint union of $v$ copies of $K_m$, which is disconnected and has just two distinct eigenvalues.
\end{remark}

Next, we show that if the total graph or the whole graph of a QSD has three distinct eigenvalues and is not regular then it must have coherent rank $9$.

\begin{lemma}
\label{lem:nonregtotal}
    Let $\mathcal Q$ be a QSD with parameters $(v,k,\lambda; b, r, \{x,y\})$.
    Then 
    \begin{itemize}
        \item[(i)] if $\mathsf T_x(\mathcal Q) \in \mathscr G_3$ is not regular then $\mathsf T_x(\mathcal Q)$ has coherent rank 9;
        \item[(ii)] if $\mathsf W_x(\mathcal Q) \in \mathscr G_3$ is not regular then $\mathsf W_x(\mathcal Q)$ has coherent rank 9.
    \end{itemize}
\end{lemma}
\begin{proof}
    We give a proof of (i).
    The proof for (ii) can be obtained mutatis mutandis.
    Suppose $\mathsf T_x(\mathcal Q) \in \mathscr G_3$ is not regular.
    Let $A = A(\Gamma)$ be the adjacency matrix of $\Gamma$, $M$ be the incidence matrix of $\mathcal Q$, and $X$ be the adjacency matrix of $\mathsf B_x(\mathcal Q)$.
    Since
        \[
    A \in \langle \left [\begin{smallmatrix}
        I & O \\ O & O
    \end{smallmatrix} \right ], \left [\begin{smallmatrix}
        J-I & O \\ O & O
    \end{smallmatrix} \right ], \left [\begin{smallmatrix}
        O & M \\ O & O
    \end{smallmatrix} \right ], \left [\begin{smallmatrix}
        O & J-M \\ O & O
    \end{smallmatrix} \right ], \left [\begin{smallmatrix}
        O & O \\ M^\transpose & O
    \end{smallmatrix} \right ], \left [\begin{smallmatrix}
        O & O \\ J-M^\transpose & O
    \end{smallmatrix} \right ],
    \left [\begin{smallmatrix}
        O & O \\ O & I
    \end{smallmatrix} \right ], \left [\begin{smallmatrix}
        O & O \\ O & X
    \end{smallmatrix} \right ],
    \left [\begin{smallmatrix}
        O & O \\ O & J-X-I
    \end{smallmatrix} \right ]\rangle,
    \]
    it is clear that the coherent closure $\mathcal {W}(\Gamma)$ has rank at most $9$.
    Observe that $\Gamma$ is biregular with valency partition $\mathfrak D(\Gamma) = \{V_1,V_2\}$ where $\{\Gamma[V_1],\Gamma[V_2]\} = \{\overline{K_{v}}, \mathsf B_x(\mathcal Q) \}$.
    By Theorem~\ref{thm:rank8}, if $\Gamma$ had coherent rank 8 then $\{\Gamma[V_1],\Gamma[V_2]\} = \{\overline{K_{v}}, K_{b} \}$, which is impossible since $\mathsf B_x(\mathcal Q)$ cannot be a complete graph.
   
    Lastly, by Theorem~\ref{thm:muzychuk}, if $\Gamma$ had coherent rank at most $7$ then $\Gamma$ would be either complete bipartite or a cone over a strongly regular graph.
    If $\Gamma$ is a complete bipartite graph then $\{\Gamma[V_1],\Gamma[V_2]\} = \{\overline{K_{v}}, \overline{K_{b}}\}$, which is impossible since $\mathsf B_x(\mathcal Q)$ cannot be an empty graph.
    Otherwise, if $\Gamma$ is a cone over a strongly regular graph then $\{\Gamma[V_1],\Gamma[V_2]\} = \{\overline{K_{1}}, \mathsf B_x(\mathcal Q)\}$, which is impossible since we must have $v \geqslant 2$.
\end{proof}

The last two results of this subsection are well-known properties of strongly regular graphs and QSDs that will be used below.
We call an eigenvalue \textbf{restricted} if it has an eigenvector perpendicular to the all-ones vector $\mathbf 1$.

\begin{lemma}[{\cite[Section 1.1.1]{srg}}]
\label{lem:SRGparamev}
    Let $\Gamma \in \operatorname{SRG}(n,k,a,c)$ have restricted eigenvalues $e_1 > e_2$.
    Then $a - c = e_1 + e_2$ and $c-k = e_1e_2$.
\end{lemma}

One more ingredient that we require is a further restriction on the parameters of a QSD.
The following necessary condition is due to Calderbank~\cite{CALDERBANK1988}.

\begin{theorem}[{\cite[Proposition 8.5.4]{srg}}]
    \label{thm:CHNineq}
    Let $\mathcal Q$ be a QSD with parameters $(v,k,\lambda;b,r,\{x,y\})$.
    Then
    \[
    (v-1)(v-2)(k-x)(k-y) - k(v-k)(v-2)(2k-x-y)+k(v-k)(k(v-k)-1) \geqslant 0,
    \]
    with equality, if and only if $\mathcal Q$ is a $3$-design.
\end{theorem}

\subsection{Underlying quasi-symmetric designs}

First, we show that if a graph in $\mathscr G_3$ has coherent rank $9$ then its coherent closure has type $\left [ \begin{smallmatrix}2 & 2 \\  & 3 \end{smallmatrix}\right ]$.

\begin{theorem}
    \label{thm:rank9}
    Let $\Gamma \in \mathscr G_3$ have coherent rank $9$.
    Then $\mathcal {W}(\Gamma)$ has type $\left [ \begin{smallmatrix}2 & 2 \\  & 3 \end{smallmatrix}\right ]$.
\end{theorem}
\begin{proof}
    By Theorem~\ref{thm:muzychuk}, the graph $\Gamma$ must not be regular.
    Using Lemma~\ref{lem:HigmanType} and Corollary~\ref{cor:allFib} (iii), we find that the coherent closure $\mathcal {W}(\Gamma)$ has type $\left [ \begin{smallmatrix}1 & 1 \\  & 6 \end{smallmatrix}\right ]$, $\left [ \begin{smallmatrix}2 & 1 \\  & 5 \end{smallmatrix}\right ]$, $\left [ \begin{smallmatrix}2 & 2 \\  & 3 \end{smallmatrix}\right ]$, $\left [ \begin{smallmatrix}3 & 1 \\  & 4 \end{smallmatrix}\right ]$, or $\left [ \begin{smallmatrix}1 & 1 & 1 \\  & 1 & 1 \\ & & 1 \end{smallmatrix}\right ]$.
    Except for $\left [ \begin{smallmatrix}2 & 2 \\  & 3 \end{smallmatrix}\right ]$, for each case with two fibres, the complement of $\Gamma$ is disconnected.
    The case with three fibres must correspond to a graph with just three vertices.
    There is just one connected graph on three vertices with three distinct eigenvalues and its coherent rank is $5$.
\end{proof}

We now show that each graph in $\mathscr G_3$ whose coherent closure has type $\left [ \begin{smallmatrix}2 & 2 \\  & 3 \end{smallmatrix}\right ]$ corresponds to a QSD.

\begin{theorem}
\label{thm:rank9nec}
Let $\Gamma \in \mathscr G_3(\theta_0,\theta_1,\theta_2)$ such that $\mathcal {W}(\Gamma)$ has type $\left [ \begin{smallmatrix}2 & 2 \\  & 3 \end{smallmatrix}\right ]$.
Then
\begin{itemize}
    \item[(i)] $\Gamma$ is biregular with degrees $k_1 \ne k_2$ and valency partition $\{V_1,V_2\}$, which is equitable with quotient matrix $\left [\begin{smallmatrix}
        k_{11} & k_{12} \\
        k_{21} & k_{22}
    \end{smallmatrix} \right ]$ and $|V_i| = n_i$ for $i \in \{1,2\}$;
    \item[(ii)] $A(\Gamma)$ has the form $\left [\begin{smallmatrix}
    \varepsilon(J-I) & M \\
    M^\transpose & A(\Gamma[V_2])
\end{smallmatrix}\right ]$,
where $\varepsilon \in \{0,1\}$ and $M$ is the incidence matrix of a QSD $\mathcal Q$ with parameters $(n_1,k_{21},\lambda;n_2,k_{12},\{x,y\})$ and $\Gamma[V_2] = \mathsf B_x(\mathcal Q)$.
\end{itemize}
Suppose that $\mathsf B_x(\mathcal Q)$ has restricted eigenvalues $e_1 > e_2$.
Then either $\theta_1 = e_1 = \frac{y-k_{21}}{x-y}$ or $\theta_2 = e_2 = \frac{y-k_{21}}{x-y}$.
Furthermore,
\begin{align}
    \label{eqn:xy}
(x,y) &= (k_{21}+(\theta_1+1)(\theta_2+1) - (e_1+1)(e_2+1),k_{21} + \theta_1\theta_2 - e_1e_2);
\end{align}
\begin{align}
 \label{eqn:rminuslam}
    \lambda-k_{12} &= \begin{cases}
        \theta_1\theta_2, & \text{ if $\Gamma[V_1]$ is empty; } \\
        (\theta_1+1)(\theta_2+1), & \text{ if $\Gamma[V_1]$ is complete; }
    \end{cases} \\
    \label{eqn:t1pt2}
   \frac{k_{12}-\lambda+y-k_{21}}{x-y}  &= \begin{cases}
        \theta_1+\theta_2, & \text{ if $\Gamma[V_1]$ is empty; } \\
        \theta_1+\theta_2+1, & \text{ if $\Gamma[V_1]$ is complete; }
    \end{cases} \\
    \label{eqn:alpha12}
    \sqrt{(k_1+\theta_1\theta_2)(k_2+\theta_1\theta_2)} &= \begin{cases}
        \frac{\lambda k_{21} - yk_{12}}{x-y}, & \text{ if $\Gamma[V_1]$ is empty; } \\
         k_{21}+\frac{\lambda k_{21} - yk_{12}}{x-y}, & \text{ if $\Gamma[V_1]$ is complete. } 
    \end{cases}
\end{align}
\end{theorem}
\begin{proof}
    Suppose $A=A(\Gamma)$ is the adjacency matrix of $\Gamma$.
By Theorem~\ref{thm:muzychuk} and Corollary~\ref{cor:allFib} (iii), the graph $\Gamma$ must be biregular with valencies $k_1$ and $k_2$ (say).
    By Corollary~\ref{cor:allFib} (iii), we can write $A = \left [\begin{smallmatrix}
    A(\Gamma[V_1]) & M \\
    M^\transpose & A(\Gamma[V_2])
\end{smallmatrix}\right ]$ and
    \[
    A^2 =\begin{bmatrix}
    A(\Gamma[V_1])^2+MM^\transpose & A(\Gamma[V_1]) M+MA(\Gamma[V_2]) \\
    M^\transpose A(\Gamma[V_1])+A(\Gamma[V_2])M^\transpose & M^\transpose M +A(\Gamma[V_2])^2
\end{bmatrix}.
\]
Apply \eqref{eqn:3ev} to obtain
\begin{equation}
\label{eqn:biregGen}
A^2
-(\theta_1+\theta_2)A
+\theta_1\theta_2 I = 
\begin{bmatrix}
    (k_1+\theta_1\theta_2) J & \sqrt{(k_1+\theta_1\theta_2)(k_2+\theta_1\theta_2)} J \\
    \sqrt{(k_1+\theta_1\theta_2)(k_2+\theta_1\theta_2)} J & (k_2+\theta_1\theta_2) J
\end{bmatrix}.
\end{equation}
By Corollary~\ref{cor:allFib} (i) and Proposition~\ref{pro:smallRank}, the subgraph $\Gamma[V_1] = K_{n_1}$ or $\overline K_{n_1}$, i.e., the matrix $A(\Gamma[V_1]) = \varepsilon_{11}(J-I)$ for some $\varepsilon_{11} \in \{0,1\}$.
Moreover, by Proposition~\ref{pro:22}, we have $\Gamma[V_2] \in \operatorname{SRG}(n_2,k_{22},a,c)$, that is, $A(\Gamma[V_2]) = X$ where
\begin{equation}
    \label{eqn:srg}
    X^2 = k_{22}I+aX+c(J-I-X).
\end{equation}

From the top-left block of \eqref{eqn:biregGen}, we obtain
\begin{equation}
\label{eqn:bbt3}
    MM^\transpose = \begin{cases}
 (k_{12}+\theta_1\theta_2)J_{n_1}-\theta_1\theta_2I, & \text{ if $\Gamma[V_1] = \overline K_{n_1}$};  \\
(k_{12}+(\theta_1+1)(\theta_2+1))J_{n_1}-(\theta_1+1)(\theta_2+1)I, & \text{ if $\Gamma[V_1] =  K_{n_1}$}.
\end{cases}
\end{equation}
Similar to the above, we can use the bottom-right block of \eqref{eqn:biregGen} together with \eqref{eqn:srg} to deduce that 
 \begin{equation*}
    M^\transpose M = 
        (k_{2}+\theta_1\theta_2-c)J_{n_2}-(k_{22}+\theta_1\theta_2-c)I+(\theta_1+\theta_2+c-a)X,
\end{equation*}
which, by Lemma~\ref{lem:SRGparamev}, becomes
\begin{equation}
\label{eqn:btb3e}
    M^\transpose M = 
        k_{21}I+(k_{21}+\theta_1\theta_2-e_1e_2)(J_{n_2}-I-X)+(k_{21}+(\theta_1+1)(\theta_2+1)-(e_1+1)(e_2+1))X,
\end{equation}
where $e_1 > e_2$ are the restricted eigenvalues of $X$.
Observe that $M$ is the incidence matrix of a quasi-symmetric $2$-design $\mathcal Q$.
Suppose $\mathcal Q$ has parameters $(v,k,\lambda;b,r,\{x,y\})$ then $v = n_1$, $k=k_{21}$, $b=n_2$, $r=k_{12}$, and $\lambda(n_1-1)=n_2(k_{21}-1)$.
Without loss of generality, we can assume that $X = A(\mathsf B_x(\mathcal Q))$.

Using \eqref{eqn:2desmmt}, we find that the eigenvalues of $MM^\transpose$ are $rk$ and $r-\lambda$ with multiplicities $1$ and $v - 1$ respectively.
Using \eqref{eqn:btb3e}, we find the eigenvalues of $M^\transpose M$ are
\begin{align*}
    & (e_1e_2-\theta_1\theta_2)+b(k+\theta_1\theta_2-e_1e_2)+e_0(\theta_1+\theta_2-e_1-e_2), \\
    & (\theta_1-e_1)(e_1-\theta_2), \text{ and } \\
    & (\theta_1-e_2)(e_2-\theta_2),
\end{align*}
with multiplicities $1$, $m_1$, and $m_2$ respectively, where $\{m_1,m_2\} = \{v-1,b-v\}$.
By Lemma~\ref{lem:LGtoL}, we must have
\begin{align*}
    0 &= \begin{cases}
        (\theta_1-e_1)(e_1-\theta_2), & \text{ if $m_1 = b-v$; } \\
    (\theta_1-e_2)(e_2-\theta_2), & \text{ if $m_2 = b-v$. }
    \end{cases}
\end{align*}
By Corollary~\ref{cor:gt0},  $\theta_1 > 0$ and $\theta_2 < 0$, whence we have either $\theta_1 = e_1 = \frac{y-k}{x-y}$ or $\theta_2 = e_2 = \frac{y-k}{x-y}$.

Comparing coefficients of \eqref{eqn:btb3e} and \eqref{eqn:mtmquasisym} yields \eqref{eqn:xy}.
Comparing coefficients of \eqref{eqn:2desmmt} and \eqref{eqn:bbt3}, we obtain \eqref{eqn:rminuslam}.
Using \eqref{eqn:mtmquasisym} together with \eqref{eqn:2desmmt}, we find that
\begin{equation}
    \label{eqn:BA2}
    MX =
    \frac{(\lambda k-yr)J  + (r-\lambda+y-k)M}{x-y}.
\end{equation}
Combining coefficients for $M$ and $J$ in the top-right block of \eqref{eqn:biregGen} yields \eqref{eqn:t1pt2} and \eqref{eqn:alpha12}.
\end{proof}

Given a graph $\Gamma \in \mathscr G_3$ such that $\mathcal {W}(\Gamma)$ has type $\left [ \begin{smallmatrix}2 & 2 \\  & 3 \end{smallmatrix}\right ]$, we call the QSD $\mathcal Q$ from the conclusion of Theorem~\ref{thm:rank9nec} the \textbf{underlying QSD} of $\Gamma$.

\begin{corollary}
\label{cor:repeatedblocks}
    Let $\Gamma \in \mathscr G_3(\theta_0,\theta_1,\theta_2)$ such that $\mathcal {W}(\Gamma)$ has type $\left [ \begin{smallmatrix}2 & 2 \\  & 3 \end{smallmatrix}\right ]$.
    Suppose that $\mathcal Q$ is the underlying QSD of $\Gamma$.
    Then $\mathcal Q$ cannot have repeated blocks.
    Furthermore, the block graph $\mathsf B_x(\mathcal Q)$ is connected.
\end{corollary}
\begin{proof}
    Suppose that $\mathcal Q$ has repeated blocks.
    Then either $k = x$ or $k = y$.
    First, suppose that $k = x$.
    Then we must have $\frac{y-k}{x-y} = -1$.
    By Theorem~\ref{thm:rank9nec}, we must have $\theta_2 = -1$, which contradicts \eqref{eqn:basicevineq}.
    Otherwise, if $k=y$ then, in a similar fashion, we find that $\theta_1 = 0$, which contradicts Corollary~\ref{cor:gt0}.
    Hence, $\mathcal Q$ does not have repeated blocks.
    Furthermore, by \cite[Theorem Q]{NeumaierRegSets}, we can deduce that the block graph $\mathsf B_x(\mathcal Q)$ is connected, as required.
\end{proof}

Corollary~\ref{cor:repeatedblocks} allows us to use the absolute bound (see Theorem~\ref{thm:abs}) to restrict our search for parametric QSDs in the next section.

\begin{theorem}
    \label{thm:converse}
    Let $\theta_1$ and $\theta_2$ be integers such that $\theta_1 \geqslant 1$ and $\theta_2 \leqslant -2$.
    Suppose $\mathcal Q$ is a QSD with parameters $(v,k,\lambda; b,r,\{x,y\})$ where $\mathsf B_x(\mathcal Q) \in \mathscr G_3(e_0,e_1,e_2)$ and
    \[
    (x,y) = \left (k+(\theta_1+1)(\theta_2+1) - (e_1+1)(e_2+1),k + \theta_1\theta_2 - e_1e_2 \right).
    \]
    \begin{itemize}
        \item[(i)] If $\lambda = r +\theta_1\theta_2$, $\frac{r-\lambda+y-k}{x-y} = \theta_1+\theta_2$, and $\frac{\lambda k-y r}{x-y} = \sqrt{\lambda(k+e_0+\theta_1\theta_2)}$ then
        \[
        \mathsf T_x(\mathcal Q) \in \mathscr G_3 \left (e_0-\theta_2,\theta_1,\theta_2 \right ).
        \]
        Furthermore, if $r \ne k+e_0$ then $\mathsf T_x(\mathcal Q)$ has coherent rank $9$.
        \item[(ii)] If $\lambda = r +(\theta_1+1)(\theta_2+1)$, $\frac{r-\lambda+y-k}{x-y} = \theta_1+\theta_2+1$, and 
        \[
        k+\frac{\lambda k-y r}{x-y} = \sqrt{(v-1+r+\theta_1\theta_2)(k+e_0+\theta_1\theta_2)}
        \]
        then 
        \[
        \mathsf W_x(\mathcal Q) \in \mathscr G_3 \left (\frac{v-1+e_0+\sqrt{(v-1-e_0)^2+4rk}}{2},\theta_1,\theta_2 \right ).
        \]
    Furthermore, if $v-1+r \ne k+e_0$ then $\mathsf W_x(\mathcal Q)$ has coherent rank $9$.
    \end{itemize}
\end{theorem}
\begin{proof}
We give a proof for (i).
    The proof for (ii) can be obtained mutatis mutandis.
    Let $A$ be the adjacency matrix of $\mathsf T_x(\mathcal Q)$, $M$ be the incidence matrix of $\mathcal Q$, and $X$ be the adjacency matrix of $\mathsf B_x(\mathcal Q)$.
    Then, we can write
    \[
A = \begin{bmatrix}
    O & M \\
    M^\transpose & X
\end{bmatrix} \quad \text{ and } \quad 
    A^2 =\begin{bmatrix}
    MM^\transpose & MX \\
    XM^\transpose & M^\transpose M +X^2
\end{bmatrix}.
\]
Next, we find
\[
X^2 = \frac{k(r-1) - y(b-1)}{x-y}I +(e_1+e_2)X + \left (\frac{k(r-1) - y(b-1)}{x-y}+e_1e_2 \right )(J-I).
\]
Let $e_0 = \frac{k(r-1) - y(b-1)}{x-y}$.
Hence, using \eqref{eqn:mtmquasisym} and \eqref{eqn:2desmmt}, we have 
    \[
    A^2 =\begin{bmatrix}
    (r-\lambda)I+\lambda J & \frac{(\lambda k-yr)J  + (r-\lambda+y-k)M}{x-y} \\
     \frac{(\lambda k-yr)J  + (r-\lambda+y-k)M^\transpose}{x-y} & (k-y-e_1e_2)I + (x-y+e_1+e_2)X+\left (e_0+e_1e_2 + y \right )J
\end{bmatrix}.
    \]
Suppose that $\lambda = r +\theta_1\theta_2$, $\frac{r-\lambda+y-k}{x-y} = \theta_1+\theta_2$, and $\frac{\lambda k-y r}{x-y} = \sqrt{\lambda(k+e_0+\theta_1\theta_2)}$.
Then 
\[
A^2 - (\theta_1+\theta_2)A + \theta_1\theta_2 I = \begin{bmatrix}
    \lambda J & \frac{r-\lambda+y-k}{x-y} J \\
    \frac{r-\lambda+y-k}{x-y} J & (k+e_0+\theta_1\theta_2)J
\end{bmatrix}.
\]
Let $\sigma = \mathfrak D(\mathsf T_x(\mathcal Q))$.
Hence, $\mathsf T_x(\mathcal Q) \in \mathscr G_3(\theta_0, \theta_1,\theta_2)$, for some $\theta_0$, which one can determine by finding the eigenvalues of $Q = Q_{\sigma}(\mathsf T_x(\mathcal Q)) = \left [ \begin{smallmatrix}
    0 & r \\
    k & e_0
\end{smallmatrix} \right ]$.
Since $\theta_0 \geqslant e_0$, the eigenvalues of $Q$ must be $\theta_0$ and $\theta_2$.
The degree sequence is $\left \{ [r]^v, [k+e_0]^b \right \}$.
Thus, if $r \ne k+e_0$ then, by Lemma~\ref{lem:nonregtotal}, $\mathsf T_x(\mathcal Q)$ has coherent rank $9$.
\end{proof}

\subsection{Parametric quasi-symmetric designs}
\label{sec:param9}

Now we apply Theorem~\ref{thm:rank9nec} to find parametrisations for quasi-symmetric designs that are the underlying QSD of a graph in $\mathscr G_3$ of coherent rank $9$.
Suppose that $\mathscr G_3$ has coherent rank $9$.
Let $\mathcal Q$ be the underlying QSD with parameters $(v,k,\lambda; b,r,\{x,y\})$ and let $e_1>e_2$ be the restricted values of $\mathsf B_x(\mathcal Q)$.
By Theorem~\ref{thm:rank9nec}, $\Gamma$ is biregular with degree sequence $\{[k_1]^{n_1},[k_2]^{n_2}\}$.
According to Theorem~\ref{thm:rank9nec}, there are four cases to consider depending on whether $\Gamma[V_1] = \overline{K_{n_1}}$ or $\Gamma[V_1] = K_{n_1}$ and $\theta_1 = e_1= \frac{y-k}{x-y}$ or $\theta_2 = e_2= \frac{y-k}{x-y}$.
We split our consideration according to each case, which we refer to as \textit{classes} as follows.

\begin{itemize}
    \item \textbf{Class 1:} $\Gamma[V_1] = \overline{K_{n_1}}$ and $\theta_2 = e_2= \frac{y-k}{x-y}$;
    \item \textbf{Class 2:} $\Gamma[V_1] = \overline{K_{n_1}}$ and $\theta_1 = e_1= \frac{y-k}{x-y}$;
    \item \textbf{Class 3:} $\Gamma[V_1] = {K_{n_1}}$ and $\theta_2 = e_2= \frac{y-k}{x-y}$;
    \item \textbf{Class 4:} $\Gamma[V_1] = {K_{n_1}}$ and $\theta_1 = e_1= \frac{y-k}{x-y}$.
\end{itemize}

\subsubsection{Class 1}

\begin{lemma}
\label{lem:case1}
    Let $\Gamma \in \mathscr G_3(\theta_0,\theta_1,\theta_2)$ with coherent rank $9$.
    Suppose that the underlying QSD $\mathcal Q$ of $\Gamma$ has parameters $(v,k,\lambda; b,r,\{x,y\})$ and $\mathsf B_x(\mathcal Q)$ has eigenvalues $e_0 > e_1 > e_2$.
    Suppose that $\Gamma[V_1]$ is empty and $\theta_2 = e_2 = \frac{y-k}{x-y}$.
    Then $\theta_0 = \theta_2(\theta_1\theta_2-\theta_1-1)$, $e_0 = \theta_1\theta_2(\theta_2-1)$, and $e_1 = \theta_1+\theta_2$.
    Furthermore, the parameters of $\mathcal Q$ can be written in terms of $\theta_1$ and $\theta_2$ according to Table~\ref{tab:case1}.
    \begin{table}[h!tbp]
        \centering
        \begin{tabular}{c|c|c|c|c|c|c}
        $v$ & $k$ & $\lambda$ & $b$ & $r$ & $x$ & $y$  \\
        \hline
        $\frac{\theta_2(\theta_2+\theta_1+\theta_1\theta_2-\theta_2^2\theta_1)}{\theta_1+1}$ & $\theta_2^2$ & $\theta_1+1$ & $\frac{(1+\theta_1-\theta_2\theta_1)(\theta_2+\theta_1+\theta_1\theta_2-\theta_2^2\theta_1)}{\theta_2(\theta_1+1)}$ & $\theta_1(1-\theta_2)+1$ & $-\theta_2$ & $0$
    \end{tabular}
        \caption{Expressions for parameters from Lemma~\ref{lem:case1}.}
        \label{tab:case1}
    \end{table}
    Conversely, suppose that, for some integers $\theta_1$ and $\theta_2$ satisfying $\theta_1 \geqslant 1$ and $\theta_2 \leqslant -2$, there exists a QSD $\mathcal Q$ whose parameters $(v,k,\lambda; b,r,\{x,y\})$ can be expressed as above in terms of $\theta_1$ and $\theta_2$.
    Then $\mathsf T_{x}(\mathcal Q) \in \mathscr G_3(\theta_2(\theta_1\theta_2-\theta_1-1),\theta_1,\theta_2)$ and $\mathsf T_x(\mathcal Q)$ has coherent rank $9$.
\end{lemma}
\begin{proof}
    Since $e_2 = \frac{y-k}{x-y}$, using Theorem~\ref{thm:blockspec}, we must have $e_1 = \frac{r-\lambda-k+y}{x-y}$.
    We use the equations \eqref{eqn:2des1}, \eqref{eqn:2des2}, \eqref{eqn:xy}, \eqref{eqn:rminuslam}, \eqref{eqn:t1pt2}, \eqref{eqn:alpha12} together with 
    Theorem~\ref{thm:blockspec} and Theorem~\ref{thm:bireg} (iii).
    Let $Q$ be the quotient matrix of the valency partition of the vertex set of $\Gamma$.
    Since the trace of $Q$ is $e_0$, which, by interlacing (Theorem~\ref{thm:interlacing}), is at most $\theta_0$, we find that $\theta_2$ is an eigenvalue of $Q$.
    Hence $e_0 = \theta_0 +\theta_2$ and $\theta_0\theta_2 = -rk$.
    The statement of the lemma follows from the above list of equations and the converse follows from Theorem~\ref{thm:converse}.
\end{proof}

Suppose that $\Gamma$ satisfies the assumptions of Lemma~\ref{lem:case1}.
The parameters of the underlying QSD of $\Gamma$ can each be expressed in terms of $\theta_1$ and $\theta_2$.
First, we consider the subfamily of QSD parameters that satisfies the additional constraint $\theta_1 = -\theta_2$.
With this additional constraint, each QSD parameter can be expressed in terms of a single variable $\theta_1$.
Indeed, each such QSD has parameters $(\theta_1^3,\theta_1^2,\theta_1+1; \theta_1(\theta_1^2+\theta_1+1), \theta_1^2+\theta_1+1, \{0,\theta_1\})$.
\begin{example}
\label{ex:rank9}
    The biregular graphs $\mathsf T_x(\mathcal Q)$ of Lemma~\ref{lem:case1} whose underlying QSD $\mathcal Q$ satisfies the additional constraint $\theta_1 = -\theta_2$ was discovered by Van Dam~\cite[Section 2.3]{VANDAM1998101}.
\end{example}

Next, instead, consider the additional constraint $\theta_1+1 = \theta_2(\theta_2-1)$.
The resulting one-parameter family violates Theorem~\ref{thm:CHNineq}.
Indeed, we have $v = -\theta_2^3+\theta_2^2+2\theta_2+1$, $k = \theta_2^2$, $\lambda = \theta_2^2+\theta_2$, $x = -\theta_2$ and $y = 0$. 
Thus, the left-hand side of the inequality of Theorem~\ref{thm:CHNineq} equals $-(\theta_2+1)^2\theta_2^2(1-\theta_2)$, which is clearly negative.

It remains to consider the QSD parameters of Lemma~\ref{lem:case1} that are not captured by the two infinite parametric families above.
By Theorem~\ref{thm:chengbipartite}, we can assume that $\theta_0 > \theta_1 \geqslant 1$ and by Theorem~\ref{thm:interlacing}, we can assume that $\theta_2 \leqslant -2$.
Furthermore, since $e_1 \geqslant 0$, we have that $\theta_1 \geqslant -\theta_2$.
In Table~\ref{tab:paramC1}, we list those for which $\theta_1 \leqslant 100$.

\subsubsection{Class 2}

Define the function $f_2(w,z)$ by
\[
f_2(w,z) := \frac{w(w^2+w+1)-2z(w+1) + w\sqrt{(w^2+w+1)^2-4wz(w+1)}}{2}.
\]

\begin{lemma}
\label{lem:case2}
    Let $\Gamma \in \mathscr G_3(\theta_0,\theta_1,\theta_2)$ with coherent rank $9$.
    Suppose that the underlying QSD $\mathcal Q$ of $\Gamma$ has parameters $(v,k,\lambda; b,r,\{x,y\})$ and $\mathsf B_x(\mathcal Q)$ has eigenvalues $e_0 > e_1 > e_2$.
    Suppose that $\Gamma[V_1]$ is empty and $\theta_1 = e_1 = \frac{y-k}{x-y}$.
    Then $\theta_0 = f_2(\theta_1,\theta_2)$, 
    $e_0 = \theta_0+\theta_2$, and $e_2 = \theta_1+\theta_2$.
    Furthermore, the parameters of $\mathcal Q$ can be written in terms of $\theta_1$ and $\theta_2$ according to Table~\ref{tab:case2}.
    \begin{table}[h!tbp]
        \centering
        \begin{tabular}{c|c|c|c|c|c|c}
        $v$ & $k$ & $\lambda$ & $b$ & $r$ & $x$ & $y$  \\
        \hline
        $\frac{-\theta_2(\theta_0-\theta_1)}{\lambda}$ & $y+\theta_1^2$ & $\frac{-\theta_2(x-1)}{\theta_1}$ & $r-k+e_0+1$ & $\lambda-\theta_1\theta_2$ & $y-\theta_1$ & $\frac{e_0-\theta_1^3+\theta_1\theta_2}{\theta_1}$
    \end{tabular}
        \caption{Expressions for parameters from Lemma~\ref{lem:case2}.}
        \label{tab:case2}
    \end{table}
    Conversely, suppose that, for some integers $\theta_1$ and $\theta_2$ satisfying $\theta_1 \geqslant 1$ and $\theta_2 \leqslant -2$, there exists a QSD $\mathcal Q$ whose parameters $(v,k,\lambda; b,r,\{x,y\})$ can be expressed as above in terms of $\theta_1$ and $\theta_2$.
    Then $\mathsf T_{x}(\mathcal Q) \in \mathscr G_3(f_2(\theta_1,\theta_2),\theta_1,\theta_2)$.
    Furthermore, if $r \ne k+f_2(\theta_1,\theta_2)+\theta_2$ then $\mathsf T_x(\mathcal Q)$ has coherent rank $9$.
\end{lemma}
\begin{proof}
Since $e_1 = \frac{y-k}{x-y}$, using Theorem~\ref{thm:blockspec}, we must have $e_2 = \frac{r-\lambda-k+y}{x-y}$.
    We use the equations \eqref{eqn:2des1}, \eqref{eqn:2des2}, \eqref{eqn:xy}, \eqref{eqn:rminuslam}, \eqref{eqn:t1pt2}, \eqref{eqn:alpha12} together with 
    Theorem~\ref{thm:blockspec} and Theorem~\ref{thm:bireg} (iii).
    Let $Q$ be the quotient matrix of the valency partition of the vertex set of $\Gamma$.
    Since the trace of $Q$ is $e_0$, which, by interlacing (Theorem~\ref{thm:interlacing}), is at most $\theta_0$, we find that $\theta_2$ is an eigenvalue of $Q$.
    Hence $e_0 = \theta_0 +\theta_2$ and $\theta_0\theta_2 = -rk$.
    The expressions for the parameters of the underlying QSD $\mathcal Q$ together with $e_0$ and $e_2$ follow from the above equations.
    We also obtain the equation
    \[
    \theta_0^2 - \theta_0(\theta_1^3 - \theta_1^2 + 2\theta_1\theta_2 - \theta_1 + 2\theta_2) - \theta_1^3\theta_2 + \theta_1^2\theta_2^2 - 2\theta_1^2\theta_2 + 2\theta_1\theta_2^2 - \theta_1\theta_2 + \theta_2^2 = 0.
    \]
    Think of the above equation as a quadratic univariate polynomial equation in $\theta_0$ with coefficients in $\mathbb Z[\theta_1,\theta_2]$ and let $r_1(\theta_1,\theta_2) \geqslant r_2(\theta_1,\theta_2)$ be its roots.
    Note that we have $r_1(\theta_1,\theta_2) = f_2(\theta_1,\theta_2)$.
    Using the fact that $x \geqslant 0$, we can rule out the possibility of $\theta_0 = r_2(\theta_1,\theta_2)$.
    Indeed, $x \geqslant 0$ implies that $\theta_0 \geqslant \theta_1^3+\theta_1^2-\theta_1\theta_2-\theta_2$.
    On the other hand, $r_2(\theta_1,\theta_2) \leqslant -\theta_2(\theta_1+1)$.
    Hence, $\theta_0 = f_2(\theta_1,\theta_2)$, as required.
    
    The converse follows from Theorem~\ref{thm:converse}.
\end{proof}

Suppose that $\Gamma$ satisfies the assumptions of Lemma~\ref{lem:case2}.
The parameters of the underlying QSD of $\Gamma$ can each be expressed in terms of $\theta_1$ and $\theta_2$.
First, we consider the subfamily of QSD parameters that satisfies the additional constraint $\theta_2 = -\theta_1^2(\theta_1^2+1)$.
With this additional constraint, each QSD parameter can be expressed in terms of a single variable $\theta_1$.

When $\theta_1 = 1$, we find that $v=b$, which implies the corresponding design is symmetric, and hence not quasi-symmetric.
When $\theta_1 = 2$, we obtain a new graph: 
\begin{example}
\label{ex:new}
    The parameters $(v,k,\lambda; b, r, \{x,y\}) = (22,15,80;176,120,\{9,11\})$ correspond to a QSD $\mathcal Q$.
    Indeed, take $\mathcal Q$ to be the complement of the QSD with parameters $(22,7,16; 176, 56, \{1,3\})$ from \cite{GS70}.
    The block graph $\mathsf B_9(\mathcal Q)$ has eigenvalues $e_0 = 70$, $e_1=2$, and $e_2 = -18$ and the total graph $\mathsf T_9(\mathcal Q)$ is a biregular graph in $\mathscr G_3(90,2,-20)$.
    Moreover,
    \[
    \operatorname{spec}\left (T_9(\mathcal Q)\right) = \left \{ \left [ 90 \right ]^1, \left [ 2 \right ]^{175}, \left [ -20 \right ]^{22} \right \}.
    \]
\end{example}

When $\theta_1 \geqslant 3$, Theorem~\ref{thm:CHNineq} is violated.
Indeed, we have 
\begin{align*}
    v&=\theta_1^3+2\theta_1^2+2\theta_1+2;  & k &= \theta_1^3+\theta_1^2+\theta_1+1; &
    x&=1+\theta_1^3; & y &= 1+\theta_1+\theta_1^3.
\end{align*}
Thus, the left-hand side of the inequality of Theorem~\ref{thm:CHNineq} is equal to 
$$-(\theta_1-2)\theta_1(\theta_1+1)^2(\theta_1^2+\theta_1+1),$$
which is clearly negative when $\theta_1 \geqslant 3$.

Next, instead, consider the additional constraint $\theta_2 = -(\theta_1^4+2\theta_1^3+\theta_1^2+\theta_1)$.
When $\theta_1 = 1$, we obtain the QSD parameters
$$(v,k,\lambda; b,r,\{x,y\}) = (7,5,10;21,15,\{3,4\})$$
and the total graph $\mathsf T_{3}(\mathcal Q) \in \operatorname{SRG}(28,15,6,10)$.
Note that the nonregularity condition $r \ne k + f_2(\theta_1,\theta_2) + \theta_2$ is violated in this case.

When $\theta_1 \geqslant 2$, Theorem~\ref{thm:CHNineq} is, again, violated.
Indeed, we have 
\begin{align*}
    v&=\theta_1^3+3\theta_1^2+2\theta_1+1;  & k &= \theta_1^3+2\theta_1^2+\theta_1+1; &
    x&=\theta_1^3+\theta_1^2+1; & y &=(\theta_1+1)(\theta_1^2+1).
\end{align*}
Thus, the left hand side of the inequality of Theorem~\ref{thm:CHNineq} equals $-\theta_1(\theta_1+1)^2(\theta_1^3-2\theta_1+1)$, which is clearly negative when $\theta_1 \geqslant 2$.

It remains to consider QSD parameters of Lemma~\ref{lem:case2} that are not captured by the two infinite parametric families above.
By Theorem~\ref{thm:chengbipartite}, we can assume that $\theta_0 > \theta_1 \geqslant 1$ and by Theorem~\ref{thm:interlacing}, we can assume that $\theta_2 \leqslant -2$.
Since $\theta_1+\theta_2 = e_2 \leqslant -2$ and $\theta_1 \geqslant 1$, we find that $1 \leqslant \theta_1 \leqslant -\theta_2-2$.
In Table~\ref{tab:paramC2}, we list those for which $\theta_2 \geqslant -100$.

\begin{remark}
\label{rem:srgs}
    Note that, if the nonregularity condition $r \ne k + f_2(\theta_1,\theta_2) + \theta_2$ is not imposed then Lemma~\ref{lem:case2} can produce QSDs whose total graph is strongly regular.
    These strongly regular graphs have an \emph{improper strongly regular decomposition} in the sense of Haemers and Higman~\cite{srdecomp}.
\end{remark}

\subsubsection{Class 3}

We need to split into two further subcases according to which of $\theta_1$ or $\theta_2$ is an eigenvalue of the quotient matrix $Q_{\mathfrak D(\Gamma)}(\Gamma)$.

Define the functions $f_3(w,z)$ and $g_3(w,z)$ by
\begin{align*}
    f_3(w,z) &:= \frac{wz(w+2)(z+2) + 2z^2 + 4z + 1  + (1-wz)\sqrt{g_3(w,z)}}{-2(z+1)(w+z+2)} \\
    g_3(w,z) &:= 5(wz+1)^2 + 4z^3(w+1) + 4wz(w+4z) +  4(z+1)(3z-1).
\end{align*}

\begin{lemma}
\label{lem:case3a}
    Let $\Gamma \in \mathscr G_3(\theta_0,\theta_1,\theta_2)$ with coherent rank $9$.
    Suppose that the underlying QSD $\mathcal Q$ of $\Gamma$ has parameters $(v,k,\lambda; b,r,\{x,y\})$ and $\mathsf B_x(\mathcal Q)$ has eigenvalues $e_0 > e_1 > e_2$.
    Suppose that $\Gamma[V_1]$ is complete, $\theta_2 = e_2 = \frac{y-k}{x-y}$, and $\theta_2$ is an eigenvalue of $Q_{\mathfrak D(\Gamma)}(\Gamma)$.
    Then $\theta_0 = f_3(\theta_1,\theta_2)$, $e_1 = \theta_1+\theta_2+1$, and
    \[
    e_0 = -\frac{\theta_2(\theta_1+1)}{(\theta_2+1)}\left ( 1 +\frac{\theta_2(1-\theta_1\theta_2)}{\theta_0+\theta_1\theta_2}\right ).
    \]
    Furthermore, the parameters of $\mathcal Q$ can be written in terms of $\theta_1$ and $\theta_2$ as follows.
    \begin{align*}
        v&= \theta_0 + \theta_2 -e_0 + 1; & b &= \frac{\theta_1(v-1)  + \theta_0}{-\theta_2}; \\
        k&= y+\theta_2(\theta_2+1); & r &= \lambda - (\theta_1+1)(\theta_2 +1); \\
        \lambda &= (\theta_1+1)\left (1+\frac{ \theta_0e_1-\theta_1e_0}{\theta_0-\theta_1}\right ); & x &= y-\theta_2-1; \\
        & & y &= \frac{-\theta_2(\lambda-e_1-1)}{\theta_1}.
    \end{align*}
    Conversely, suppose that, for some integers $\theta_1$ and $\theta_2$ satisfying $\theta_1 \geqslant 1$ and $\theta_2 \leqslant -2$, there exists a QSD $\mathcal Q$ whose parameters $(v,k,\lambda; b,r,\{x,y\})$ can be expressed as above in terms of $\theta_1$ and $\theta_2$.
    Then $\mathsf W_{x}(\mathcal Q) \in \mathscr G_3(f_3(\theta_1,\theta_2),\theta_1,\theta_2)$ and $\mathsf W_x(\mathcal Q)$ has coherent rank $9$.
\end{lemma}
\begin{proof}
Since $e_2 = \frac{y-k}{x-y}$, using Theorem~\ref{thm:blockspec}, we must have $e_1 = \frac{r-\lambda-k+y}{x-y}$.
    We use the equations \eqref{eqn:2des1}, \eqref{eqn:2des2}, \eqref{eqn:xy}, \eqref{eqn:rminuslam}, \eqref{eqn:t1pt2}, \eqref{eqn:alpha12} together with 
    Theorem~\ref{thm:blockspec} and Theorem~\ref{thm:bireg} (iii).
    Using the trace and determinant of $Q_{\mathfrak D(\Gamma)}(\Gamma)$, we obtain the equations
    $v-1+e_0 =\theta_0 +\theta_2$ and $\theta_0\theta_2 = (v-1)e_0-rk$.
    The expressions for the parameters of the underlying QSD $\mathcal Q$ together with $e_0$ and $e_1$ follow from the above equations.
    We also obtain the equation
    \[
    \xi_2(\theta_1,\theta_2)\theta_0^2 + \xi_1(\theta_1,\theta_2)\theta_0 +\xi_0(\theta_1,\theta_2) = 0,
    \]
    where
    \begin{align*}
\xi_2(\theta_1,\theta_2)&:=\theta_1\theta_2+\theta_1+\theta_2^2+3\theta_2+2; \\
    \xi_1(\theta_1,\theta_2)&:=\theta_1^2\theta_1^2 + 2\theta_1^2\theta_2 + 2\theta_1\theta_2^2 + 4\theta_1\theta_2+2\theta_2^2+4\theta_2+1; \\
        \xi_0(\theta_1,\theta_2)&:=- \theta_1^3\theta_2^3 + \theta_1^3\theta_2^2 - \theta_1^2\theta_2^3 + 4\theta_1^2\theta_2^2 + 4\theta_1\theta_2^2 + \theta_2^2.
    \end{align*}
    Think of the above equation as a quadratic univariate polynomial equation in $\theta_0$ with coefficients in $\mathbb Z[\theta_1,\theta_2]$ and let $s_1(\theta_1,\theta_2) \geqslant s_2(\theta_1,\theta_2)$ be its roots.
    Note that we have $s_1(\theta_1,\theta_2) = f_3(\theta_1,\theta_2)$.
    Using the fact that $\theta_0 > \theta_1 \geqslant 0$, we can rule out the possibility of $\theta_0 = s_2(\theta_1,\theta_2)$.
    Suppose to the contrary that $\theta_0 = s_2(\theta_1,\theta_2)$.
    Then $\theta_0 > \theta_1 \geqslant 0$ and $\theta_2 < 0$ implies that 
    \[
    \theta_1\theta_2(\theta_1+2)(\theta_2+2) + 2\theta_2^2 + 4\theta_2 + 1  - (1-\theta_1\theta_2)\sqrt{g_3(\theta_1,\theta_2)} > -2(\theta_2+1)(\theta_1+\theta_2+2)\theta_1,
    \]
    which implies
    \begin{align*}
        & (\theta_1\theta_2(\theta_1+2)(\theta_2+2) + 2\theta_2^2 + 4\theta_2 + 1 + 2(\theta_2+1)(\theta_1+\theta_2+2)\theta_1)^2-(1-\theta_1\theta_2)^2 g_3(\theta_1,\theta_2)\\
        = & -4(\theta_2+1)(\theta_1+1)(\theta_1+\theta_2+2)(\theta_2\theta_1^2(\theta_2-3)(\theta_2+1)-\theta_1\theta_2(5\theta_2+4)-\theta_2^2-\theta_1^2-\theta_1)
    \end{align*}
    is positive.
    Now, $-4(\theta_2+1)(\theta_1+1)(\theta_1+\theta_2+2)$ is positive since $e_1=\theta_1+\theta_2+1 \geqslant 0$, $\theta_2 \leqslant -2$, and $\theta_1\geqslant 0$.
    However, it also follows that $\theta_2\theta_1^2(\theta_2-3)(\theta_2+1)-\theta_1\theta_2(5\theta_2+4)-\theta_2^2-\theta_1^2-\theta_1$ is negative, which contradicts our supposition.
    Hence, $\theta_0 = f_3(\theta_1,\theta_2)$, as required.

    The converse follows from Theorem~\ref{thm:converse}.
\end{proof}

Suppose that $\Gamma$ satisfies the assumptions of Lemma~\ref{lem:case3a}.
By Theorem~\ref{thm:chengbipartite}, we can assume that $\theta_0 > \theta_1 \geqslant 1$ and by Theorem~\ref{thm:interlacing}, we can assume that $\theta_2 \leqslant -2$.
Since $\theta_1+\theta_2+1 = e_1 \geqslant 0$ and $\theta_2 \leqslant -2$, we find that $-\theta_1-1 \leqslant \theta_2 \leqslant -2$.
In Table~\ref{tab:paramC3}, we list all the parameters for QSDs corresponding to Lemma~\ref{lem:case3a} for $\theta_1 \leqslant 100$.
\begin{example}
\label{ex:MK}
    When $\theta_1 = 5$ and $\theta_2 = -2$, using the expressions from Lemma~\ref{lem:case3a}, we obtain the QSD parameters $(8,6,15; 28,21,\{4,5\})$.
    A QSD $\mathcal Q$ having such parameters exists whose blocks are all $6$-sets of $\{1,2,\dots,8\}$.
    The block graph $\mathsf B_5(\mathcal Q)$ has eigenvalues $(e_0,e_1,e_2) = (12,4,-2)$ and the whole graph $\mathsf W_5(\mathcal Q)$ is a biregular graph in $\mathscr G_3(21,5,-2)$.
    This graph was found by Muzychuk and Klin~\cite{MUZYCHUK1998191}.
\end{example}

Define the polynomial $h(u,w,z)$ as
\[
h(u,w,z) := C_3(w,z) u^3+C_2(w,z) u^2 + C_1(w,z) u+C_0(w,z),
    \]
    where
    \begin{align*}
        C_3(w,z)&= (w+z)(w+1)(z+1); \\
        C_2(w,z) &= w^4 + w^3(z+2)^2 + w^2(z^2 + 7z + 5) - w(2z^3 + 3z^2 - z - 1) - 2z^2(z+1); \\
        C_1(w,z) &= (2w^5+z^3)(z+1) - w^4(z^3 + z^2 - 4z - 6) - w^3(z^3 - 6z - 8) \\
        &\ \ \ - 3w^2(z^3 + 3z^2 + z+1)  +  wz(z^3 - 5z - 3); \\
        C_0(w,z) &= w(w+1)(w-z)(w^3(z + 1)^2 + 2w^2(z+1) + (w-z)(z+1)(z+2)-wz).
    \end{align*}

\begin{lemma}
\label{lem:case3b}
    Let $\Gamma \in \mathscr G_3(\theta_0,\theta_1,\theta_2)$ with coherent rank $9$.
    Suppose that the underlying QSD $\mathcal Q$ has parameters $(v,k,\lambda; b,r,\{x,y\})$ and $\mathsf B_x(\mathcal Q)$ has eigenvalues $e_0 > e_1 > e_2$.
    Suppose that $\Gamma[V_1]$ is complete, $\theta_2 = e_2= \frac{y-k}{x-y}$, and $\theta_1$ is an eigenvalue of $Q_{\mathfrak D(\Gamma)}(\Gamma)$.
    Then $\theta_0$ is the positive zero of the (univariate) polynomial $p(u):= h(u,\theta_1,\theta_2)$.
    Furthermore,
    \[
    e_1 = \theta_1 + \theta_2+1, \text{ and } e_0 = \frac{(\theta_2+1)(\theta_0(\theta_1 e_1 -1) +\theta_1^2(\theta_2+1))}{\theta_0(\theta_1\theta_2 + e_1) + \theta_1^3(\theta_2+1)+2\theta_1(\theta_1+1)-\theta_2(\theta_2+2)}.
    \]
    The parameters of $\mathcal Q$ can be written in terms of $\theta_0$, $\theta_1$ and $\theta_2$ as follows.
    \begin{align*}
        v&= \theta_0 + \theta_1 -e_0 + 1; & b &= \frac{\theta_1v  - \theta_0+\theta_2}{\theta_2}; \\
        k&= y+\theta_2(\theta_2+1); & r &= \lambda - (\theta_1+1)(\theta_2 +1); \\
        \lambda &= \theta_2+1-\frac{\theta_1(y-e_1-1)}{\theta_2}; & x &= y-\theta_2-1; \\
        & & y &= (\theta_2+1)\frac{ \theta_0(e_1-1)-\theta_1(e_0-t_1-1)-\theta_2^2}{\theta_0-\theta_2}.
    \end{align*}
    Conversely, suppose that $\theta_0$ is an integer solution to $p(u) = 0$ for some integers $\theta_1$ and $\theta_2$ satisfying $\theta_0>\theta_1 \geqslant 1$ and $\theta_2 \leqslant -2$.
    If there exists a QSD $\mathcal Q$ whose parameters $(v,k,\lambda; b,r,\{x,y\})$ can be expressed as above in terms of $\theta_0$, $\theta_1$, and $\theta_2$ then $\mathsf W_{x}(\mathcal Q) \in \mathscr G_3(\theta_0,\theta_1,\theta_2)$.
    Suppose $\mathsf B_x(\mathcal Q)$ has degree $e_0$.
    If $v-1+r \ne k+e_0$ then $\mathsf W_x(\mathcal Q)$ has coherent rank $9$.
\end{lemma}
\begin{proof}
Since $e_2 = \frac{y-k}{x-y}$, using Theorem~\ref{thm:blockspec}, we must have $e_1 = \frac{r-\lambda-k+y}{x-y}$.
    We use the equations \eqref{eqn:2des1}, \eqref{eqn:2des2}, \eqref{eqn:xy}, \eqref{eqn:rminuslam}, \eqref{eqn:t1pt2}, \eqref{eqn:alpha12} together with 
    Theorem~\ref{thm:blockspec} and Theorem~\ref{thm:bireg} (iii).
    Using the trace and determinant of $Q_{\mathfrak D(\Gamma)}(\Gamma)$, we obtain the equations
    $v-1+e_0 =\theta_0 +\theta_1$ and $\theta_0\theta_1 = (v-1)e_0-rk$.
    The expressions for the parameters of the underlying QSD $\mathcal Q$ together with $e_0$ and $e_1$ follow from the above equations.
    Now we show that the polynomial $p(u)$ has precisely one real zero.
    Fix $\theta_1 \geqslant 1$ and $\theta_2 \leqslant -2$.
    If $\theta_1 = -\theta_2-1$ then it is easy to check that $p(u) = 0$ has no positive roots. 
    Otherwise, we can assume that $\theta_1 > -\theta_2-1$.
    Since $C_3(\theta_1,\theta_2) < 0$, $C_2(\theta_1,\theta_2) < 0$, $C_1(\theta_1,\theta_2) < 0$, and $C_3(\theta_1,\theta_2) > 0$, by Descartes' rule of signs, we have that $p(u)$ has just one positive zero.

    The converse follows from Theorem~\ref{thm:converse}.
\end{proof}

There do not appear to be any QSD parameters that satisfy Lemma~\ref{lem:case3b}.
However, if we do not impose the nonregularity condition $v-1+r \ne k+e_0$ then one can obtain QSDs whose whole graph is strongly regular.
These strongly regular graphs correspond to the complement of those of Remark~\ref{rem:srgs}.

\subsubsection{Class 4}

\begin{lemma}
\label{lem:case4a}
    Let $\Gamma \in \mathscr G_3(\theta_0,\theta_1,\theta_2)$ with coherent rank $9$.
    Suppose that the underlying QSD $\mathcal Q$ has parameters $(v,k,\lambda; b,r,\{x,y\})$ and $\mathsf B_x(\mathcal Q)$ has eigenvalues $e_0 > e_1 > e_2$.
    Suppose that $\Gamma[V_1]$ is complete, $\theta_1 = e_1= \frac{y-k}{x-y}$, and $\theta_2$ is an eigenvalue of $Q_{\mathfrak D(\Gamma)}(\Gamma)$.
    Then $\theta_0$ is a zero of the (univariate) polynomial $q(u) := h(u,\theta_2,\theta_1)$.
    Furthermore,
    \[
    e_2 = \theta_1 + \theta_2+1, \text{ and } e_0 = \frac{(\theta_1+1)(\theta_0(\theta_2 e_2 -1) +\theta_2^2(\theta_1+1))}{\theta_0(\theta_1\theta_2 + e_2) + \theta_2^3(\theta_1+1)+2\theta_2(\theta_2+1)-\theta_1(\theta_1+2)}.
    \]
    The parameters of $\mathcal Q$ can be written in terms of $\theta_0$, $\theta_1$, and $\theta_2$ as follows.
    \begin{align*}
        v&= \theta_0 + \theta_2 -e_0 + 1; & b &= \frac{\theta_2v  - \theta_0+\theta_1}{\theta_1}; \\
        k&= y+\theta_1(\theta_1+1); & r &= \lambda - (\theta_1+1)(\theta_2 +1); \\
        \lambda &= \theta_1+1-\frac{\theta_2(y-e_2-1)}{\theta_1}; & x &= y-\theta_1-1; \\
        & & y &= (\theta_1+1)\frac{ \theta_0(e_2-1)-\theta_1(e_0-t_2-1)-\theta_1^2}{\theta_0-\theta_1}.
    \end{align*}
    Conversely, suppose that $\theta_0$ is an integer solution to $q(u) = 0$ for some integers $\theta_1$ and $\theta_2$ satisfying $\theta_0>\theta_1> 0$ and $\theta_2 \leqslant -2$.
    If there exists a QSD $\mathcal Q$ whose parameters $(v,k,\lambda; b,r,\{x,y\})$ can be expressed as above in terms of $\theta_0$, $\theta_1$, and $\theta_2$ then $\mathsf W_{x}(\mathcal Q) \in \mathscr G_3(\theta_0,\theta_1,\theta_2)$.
    Suppose $\mathsf B_x(\mathcal Q)$ has degree $e_0$.
    If $v-1+r \ne k+e_0$ then $\mathsf W_x(\mathcal Q)$ has coherent rank $9$.
\end{lemma}
\begin{proof}
    Since $e_1 = \frac{y-k}{x-y}$, using Theorem~\ref{thm:blockspec}, we must have $e_2 = \frac{r-\lambda-k+y}{x-y}$.
    We use the equations \eqref{eqn:2des1}, \eqref{eqn:2des2}, \eqref{eqn:xy}, \eqref{eqn:rminuslam}, \eqref{eqn:t1pt2}, \eqref{eqn:alpha12} together with 
    Theorem~\ref{thm:blockspec} and Theorem~\ref{thm:bireg} (iii).
    Using the trace and determinant of $Q_{\mathfrak D(\Gamma)}(\Gamma)$, we obtain the equations
    $v-1+e_0 =\theta_0 +\theta_2$ and $\theta_0\theta_2 = (v-1)e_0-rk$.
    The system of equations can be obtained from that of the proof of Lemma~\ref{lem:case3b} by interchanging $\theta_1$ with $\theta_2$ and $e_1$ with $e_2$.
    
    The converse follows from Theorem~\ref{thm:converse}.
\end{proof}

We are not aware of any QSD parameters that satisfy Lemma~\ref{lem:case4a}.

\begin{lemma}
\label{lem:case4b}
    Let $\Gamma \in \mathscr G_3(\theta_0,\theta_1,\theta_2)$ with coherent rank $9$.
    Suppose that the underlying QSD $\mathcal Q$ of $\Gamma$ has parameters $(v,k,\lambda; b,r,\{x,y\})$ and $\mathsf B_x(\mathcal Q)$ has eigenvalues $e_0 > e_1 > e_2$.
    Suppose that $\Gamma[V_1]$ is complete, $\theta_1 = e_1 = \frac{y-k}{x-y}$, and $\theta_1$ is an eigenvalue of $Q_{\mathfrak D(\Gamma)}(\Gamma)$.
    Then $\theta_0 = f_3(\theta_2,\theta_1)$, $e_2 = \theta_1+\theta_2+1$, and
    \[
    e_0 = \frac{\theta_1(\theta_2+1)}{(\theta_1+1)}\left ( 1 +\frac{\theta_1(1-\theta_1\theta_2)}{\theta_0+\theta_1\theta_2}\right ).
    \]
    Furthermore, the parameters of $\mathcal Q$ can be written in terms of $\theta_1$ and $\theta_2$ as follows.
    \begin{align*}
        v&= \theta_0 + \theta_1 -e_0 + 1; & b &= \frac{\theta_2(v-1)  + \theta_0}{-\theta_1}; \\
        k&= y+\theta_1(\theta_1+1); & r &= \lambda - (\theta_1+1)(\theta_2 +1); \\
        \lambda &= (\theta_2+1)\left (1+\frac{ \theta_0e_2-\theta_2e_0}{\theta_0-\theta_2}\right ); & x &= y-\theta_1-1; \\
        & & y &= \frac{-\theta_1(\lambda-e_2-1)}{\theta_2}.
    \end{align*}
    Conversely, suppose that, for some integers $\theta_1$ and $\theta_2$ satisfying $\theta_1 \geqslant 1$ and $\theta_2 \leqslant -2$, there exists a QSD $\mathcal Q$ whose parameters $(v,k,\lambda; b,r,\{x,y\})$ can be expressed as above in terms of $\theta_1$ and $\theta_2$.
    Then $\mathsf W_{x}(\mathcal Q) \in \mathscr G_3(f_3(\theta_2,\theta_1),\theta_1,\theta_2)$ and $\mathsf W_x(\mathcal Q)$ has coherent rank $9$.
\end{lemma}
\begin{proof}
Since $e_1 = \frac{y-k}{x-y}$, using Theorem~\ref{thm:blockspec}, we must have $e_2 = \frac{r-\lambda-k+y}{x-y}$.
    We use the equations \eqref{eqn:2des1}, \eqref{eqn:2des2}, \eqref{eqn:xy}, \eqref{eqn:rminuslam}, \eqref{eqn:t1pt2}, \eqref{eqn:alpha12} together with 
    Theorem~\ref{thm:blockspec} and Theorem~\ref{thm:bireg} (iii).
    Using the trace and determinant of $Q_{\mathfrak D(\Gamma)}(\Gamma)$, we obtain the equations
    $v-1+e_0 =\theta_0 +\theta_1$ and $\theta_0\theta_1 = (v-1)e_0-rk$.
    The system of equations can be obtained from that of the proof of Lemma~\ref{lem:case3a} by interchanging $\theta_1$ with $\theta_2$ and $e_1$ with $e_2$.
    In particular, we obtain the equation
    \[
    \xi_2(\theta_2,\theta_1)\theta_0^2 + \xi_1(\theta_2,\theta_1)\theta_0 +\xi_0(\theta_2,\theta_1) = 0,
    \]
    where $\xi_0$, $\xi_1$, and $\xi_2$ are as defined in the proof of Lemma~\ref{lem:case3a}.
    Think of the above equation as a quadratic univariate polynomial equation in $\theta_0$ with coefficients in $\mathbb Z[\theta_1,\theta_2]$ and let $s_1(\theta_2,\theta_1) \geqslant s_2(\theta_2,\theta_1)$ be its roots.
    Note that $s_1(\theta_2,\theta_1) = f_3(\theta_2,\theta_1)$.
    Using the fact that $\theta_0 > \theta_1 \geqslant 0$, we can rule out the possibility of $\theta_0 = s_2(\theta_2,\theta_1)$.
    Suppose to the contrary that $\theta_0 = s_2(\theta_2,\theta_1)$.
    Then $\theta_0 > \theta_1 \geqslant 0$ and $\theta_2 < 0$ implies that 
    \[
    \theta_1\theta_2(\theta_1+2)(\theta_2+2) + 2\theta_1^2 + 4\theta_1 + 1  - (1-\theta_1\theta_2)\sqrt{g_3(\theta_2,\theta_1)} > -2(\theta_1+1)(\theta_1+\theta_2+2)\theta_1,
    \]
    which implies
    \begin{align*}
        & (\theta_1\theta_2(\theta_1+2)(\theta_2+2) + 2\theta_1^2 + 4\theta_1 + 1 + 2(\theta_1+1)(\theta_1+\theta_2+2)\theta_1)^2-(1-\theta_1\theta_2)^2 g_3(\theta_2,\theta_1)\\
        = & -4\theta_1(\theta_1+1)(\theta_1+\theta_2+2)(\theta_1\theta_2(\theta_2^2(\theta_1-1)-3(\theta_1+2\theta_2+3))-\theta_1^3-5\theta_1^2-7\theta_1-1)
    \end{align*}
    is positive.
    Now, $-4\theta_1(\theta_1+1)(\theta_1+\theta_2+2)$ is positive since $\theta_1\geqslant 0$ and $e_2=\theta_1+\theta_2+1 \leqslant -1$.
    However, together with the condition that $\theta_2 \leqslant -2$, it also follows that 
    $$\theta_1\theta_2(\theta_2^2(\theta_1-1)-3(\theta_1+2\theta_2+3))-\theta_1^3-5\theta_1^2-7\theta_1-1$$ 
    is negative, which contradicts our supposition.
    Hence, $\theta_0 = f_3(\theta_2,\theta_1)$, as required.
        
    The converse follows from Theorem~\ref{thm:converse}.
\end{proof}

Suppose that $\Gamma$ satisfies the assumptions of Lemma~\ref{lem:case4b}.
By Theorem~\ref{thm:chengbipartite}, we can assume that $\theta_0 > \theta_1 \geqslant 1$ and by Theorem~\ref{thm:interlacing}, we can assume that $\theta_2 \leqslant -2$.
Since $\theta_1+\theta_2 +1 = e_2 \leqslant -2$ and $\theta_1 \geqslant 1$, we find that $1 \leqslant \theta_1 \leqslant -\theta_2-3$.
In Table~\ref{tab:paramC4}, we list all the parameters for QSDs corresponding to Lemma~\ref{lem:case3a} for $\theta_2 \geqslant -100$.

\section{Triregular graphs with small coherent rank}

\label{sec:tri}

In this section, we establish a lower bound for the coherent rank of a triregular graph in $\mathscr G_3$.
This bound has the potential to be sharp (see Example~\ref{ex:14}).

\begin{theorem}
\label{thm:rank14lb}
    Let $\Gamma \in \mathscr G_3$ be a triregular graph.
    Then the coherent closure $\mathcal {W}(\Gamma)$ has rank at least $14$.
\end{theorem}

We will prove Theorem~\ref{thm:rank14lb} after we state and prove a series of three lemmas that restrict the entries of the type matrix.

\begin{lemma}
\label{lem:type211}
Let $\Gamma \in \mathscr G_3$ be a triregular graph whose coherent closure $\mathcal {W}(\Gamma)$ has type $ \left [\begin{smallmatrix} 2 & t_{12} & t_{13} \\  & t_{22} & t_{23} \\ &  & t_{33} \end{smallmatrix} \right ]$.
Then $t_{12}t_{13} > 1$.
\end{lemma}

    \begin{proof}
Suppose, for a contradiction, that the coherent closure $\mathcal {W}(\Gamma)$ has type $ \left [\begin{smallmatrix} 2 & 1 & 1 \\  & t_{22} & t_{23} \\  &  & t_{33} \end{smallmatrix} \right ]$ and suppose that $\mathfrak D(\Gamma) = \{V_1,V_2,V_3\}$.
Then, by Corollary~\ref{cor:allFib} (iii), the adjacency matrix $A$ of $\Gamma$ has the form  
\[
A = \begin{bmatrix}
A(\Gamma[V_1]) & \varepsilon_{2} J & \varepsilon_{3} J  \\
\varepsilon_{2} J & A(\Gamma[V_2]) & \star \\
\varepsilon_{3} J & \star & A(\Gamma[V_3])
\end{bmatrix},
\]
where $\varepsilon_1,\varepsilon_2,\varepsilon_3 \in \{0,1\}$.
Suppose that $X = A(\Gamma[V_1]) = \varepsilon_1(J-I)$ has order $n_1$ and $\Gamma[V_2]$ and $\Gamma[V_3]$ have orders $n_2$ and $n_3$, respectively.
Then
\[
A^2 = \begin{bmatrix}
X^2+\varepsilon_{2}n_2 J +\varepsilon_{3}n_3J & \star & \star  \\
\star & \star & \star \\
\star & \star & \star
\end{bmatrix}.
\]
    Now, the top-left block of \eqref{eqn:3ev} becomes
\begin{align*}
    X^2+\varepsilon_{2}n_2J +\varepsilon_{3}n_3J  - (\theta_1+\theta_2)X+\theta_1\theta_2I &= (\varepsilon_{1}(n_1-1)+n_2+n_3+\theta_1\theta_2)J.
\end{align*}

In the case where $X = O$, we have $\theta_1\theta_2 = 0$, but this means that $\theta_1 = 0$ which, by Theorem~\ref{thm:chengbipartite}, implies that $\Gamma$ is a complete bipartite graph, which is a contradiction.

In the case where $X = J-I$, we have \begin{align*}
(J-I)^2 - (\theta_1+\theta_2) (J-I) + \theta_1\theta_2 I &= O \\
J^2 - 2J -(\theta_1+\theta_2)J+ I + (\theta_1+\theta_2) I + \theta_1\theta_2 I &= O
\end{align*}
which means $(\theta_1+1)(\theta_2 + 1)  = 0$, but this contradicts \eqref{eqn:basicevineq}.
\end{proof}

Next, we restrict the product of the off-diagonal entries of the type matrix.

\begin{lemma}
\label{lem:type31113}
Let $\Gamma \in \mathscr G_3(\theta_0,\theta_1,\theta_2)$ be a triregular graph whose coherent closure $\mathcal {W}(\Gamma)$ has type $ \left [\begin{smallmatrix} t_{11} & t_{12} & t_{13} \\  & t_{22} & t_{23} \\ &  & t_{33} \end{smallmatrix} \right ]$.
Then $t_{12}t_{13}t_{23} > 1$.
\end{lemma}

\begin{proof}
Suppose, for a contradiction, that $\mathcal {W}(\Gamma)$ has type $ \left [\begin{smallmatrix} t_{11} & 1 & 1 \\  & t_{22} & 1 \\  &  & t_{33} \end{smallmatrix} \right ]$ and suppose that $\mathfrak D(\Gamma) = \{V_1,V_2,V_3\}$.
Then, by Corollary~\ref{cor:allFib} (iii), the adjacency matrix $A$ of $\Gamma$ has the form  
\[
A = \begin{bmatrix}
A(\Gamma[V_1]) & \varepsilon_{12} J &  \varepsilon_{13} J  \\
 \varepsilon_{12} J  & A(\Gamma[V_2]) & \varepsilon_{23} J  \\
\varepsilon_{13} J  & \varepsilon_{23} J  & A(\Gamma[V_3])
\end{bmatrix},
\]
where $\varepsilon_{12},\varepsilon_{13},\varepsilon_{23} \in \{0,1\}$ and $\Gamma[V_1]$, $\Gamma[V_2]$, and $\Gamma[V_3]$ are regular graphs orders $n_1$, $n_2$, and $n_3$, respectively.
First, suppose that $\varepsilon_{12}\varepsilon_{13}\varepsilon_{23} = 0$.
Without loss of generality, we can assume that $\varepsilon_{23} = 0$.
Since $\Gamma$ is connected, we must have $\varepsilon_{12}=\varepsilon_{13} = 1$.
Thus, the complement of $\Gamma$ is disconnected.
By Theorem~\ref{thm:cheng}, $\Gamma$ must be a cone.
Hence, $\Gamma[V_1] = K_{n_1}$ and by equating coefficients of $I$ in the $(1,1)$-block of \eqref{eqn:3ev} yields $(\theta_1+1)(\theta_2 + 1)  = 0$, which contradicts \eqref{eqn:basicevineq}.

Lastly, suppose $\varepsilon_{12}=\varepsilon_{13}=\varepsilon_{23}=1$.
The same argument establishes a contradiction, as required.
\end{proof}

Our final lemma rules out the possibility of the type matrix $ \left [\begin{smallmatrix} 1 & 1 & 1 \\  & 2 & 2 \\ &  & 2 \end{smallmatrix} \right ]$.

\begin{lemma}
\label{lem:conesd}
Let $\Gamma \in \mathscr G_3(\theta_0,\theta_1,\theta_2)$ be a triregular graph whose coherent closure $\mathcal {W}(\Gamma)$ has type $ \left [\begin{smallmatrix} t_{11} & 1 & 1 \\  & 2 & 2 \\ &  & 2 \end{smallmatrix} \right ]$.
Then $t_{11} > 1$.
\end{lemma}
\begin{proof}
    Suppose, for a contradiction, that the coherent closure $\mathcal {W}(\Gamma)$ has type $\left [ \begin{smallmatrix}1 & 1 & 1 \\  & 2 & 2 \\  &  & 2 \end{smallmatrix}\right ]$ and suppose that $\mathfrak D(\Gamma) = \{V_1,V_2,V_3\}$.
Then, by Proposition~\ref{pro:smallRank} and Corollary~\ref{cor:allFib} (iii), the adjacency matrix $A$ of $\Gamma$ has the form  
\[
A = \begin{bmatrix}
    0 &  k_{21}\mathbf 1^\transpose & k_{31}\mathbf 1^\transpose \\
    k_{21}\mathbf 1 & A(\Gamma[V_2]) & M \\
    k_{31}\mathbf 1 & M^\transpose & A(\Gamma[V_3])
\end{bmatrix},
\]
where $\Gamma[V_1] = K_1$ and $\Gamma[V_2]$ and $\Gamma[V_3]$ are both empty or complete graphs of orders $n_2$ and $n_3$, respectively.
Let $A_2 = A(\Gamma[V_2])$, $A_3 = A(\Gamma[V_3])$, let $k_{12} = k_{21}n_2$, $k_{13}=k_{31}n_3$, and let $k_{22}$, $k_{23}$, $k_{32}$, and $k_{33}$ satisfy $A_2 \mathbf 1 = k_{22} \mathbf 1$, $M \mathbf 1 = k_{23} \mathbf 1$, $M^\transpose \mathbf 1 = k_{32} \mathbf 1$, and $A_3 \mathbf 1 = k_{33} \mathbf 1$.
Then
\[
A^2 =     \begin{bmatrix}
     k_{12}  + k_{13} & (k_{21}k_{22} + k_{31}k_{23})\mathbf 1^\transpose & (k_{21}k_{32} + k_{31}k_{33}) \mathbf 1^\transpose \\
    (k_{21} k_{22} + k_{31}k_{23}) \mathbf 1 & k_{21}J + A_2^2+MM^\transpose & k_{21}k_{31}J + A_2 M+MA_3 \\
    (k_{21} k_{32} + k_{31}k_{33}) \mathbf 1 & k_{21}k_{31}J + M^\transpose A_2+A_3M^\transpose & k_{31}J  + M^\transpose M +A_3^2
\end{bmatrix}.
\]
Apply \eqref{eqn:3ev} to obtain
\begin{equation}
\label{eqn:triregGen}
    A^2
-(\theta_1+\theta_2)A
+\theta_1\theta_2 I = 
\begin{bmatrix}
    \alpha_1^2 J  & \alpha_1 \alpha_2 J & \alpha_1 \alpha_3 J \\
    \alpha_1 \alpha_2 J & \alpha_2^2 J & \alpha_2 \alpha_3 J \\
    \alpha_1 \alpha_3 J &  \alpha_2 \alpha_3 J &  \alpha_3^2 J
\end{bmatrix}, \text{ where }
\end{equation}
$\alpha_1 = \sqrt{k_{12}  + k_{13}+\theta_1\theta_2}$, $\alpha_2 = \sqrt{k_{21}  + k_{22}+ k_{23}+\theta_1\theta_2}$, and $\alpha_3=\sqrt{k_{31}  + k_{32}+ k_{33}+\theta_1\theta_2}$.

From the centre block of \eqref{eqn:triregGen}, we obtain
\begin{equation}
\label{eqn:bbt22}
    MM^\transpose = \begin{cases}
 (k_{23}+\theta_1\theta_2)J_{n_2}-\theta_1\theta_2I, & \text{ if $\Gamma[V_2] = \overline K_{n_2}$};  \\
(k_{23}+(\theta_1+1)(\theta_2+1))J_{n_2}-(\theta_1+1)(\theta_2+1)I, & \text{ if $\Gamma[V_2] =  K_{n_2}$}.
\end{cases}
\end{equation}
Similarly, we can use the bottom-right block of \eqref{eqn:triregGen} to deduce that 
\begin{equation}
\label{eqn:bbt33}
    M^\transpose M = \begin{cases}
 (k_{32}+\theta_1\theta_2)J_{n_3}-\theta_1\theta_2I, & \text{ if $\Gamma[V_3] = \overline K_{n_3}$};  \\
(k_{32}+(\theta_1+1)(\theta_2+1))J_{n_3}-(\theta_1+1)(\theta_2+1)I, & \text{ if $\Gamma[V_3] =  K_{n_3}$}.
\end{cases}
\end{equation}

Using \eqref{eqn:bbt22} and \eqref{eqn:bbt33} together with Lemma~\ref{lem:LGtoL}, we find that $n_2 = n_3$.
By double counting edges between $V_2$ and $V_3$, we find that $k_{23} = k_{32}$.

First, we assume that both $\Gamma[V_2]$ and $\Gamma[V_3]$ are empty, that is, $k_{22} = 0$ and $k_{33} = 0$.
Since $\Gamma$ is triregular, we cannot have both $k_{21} = k_{31} = 1$.
Without loss of generality, since $\Gamma$ is connected, we assume that $k_{21}=1$ and $k_{31}=0$.
Then the valency-partition $\mathfrak D(\Gamma)$, which, by Theorem~\ref{thm:equitableValency}, is equitable, has quotient matrix
\[
Q = Q_{\mathfrak D(\Gamma)}(\Gamma) = \begin{bmatrix}
    0 & n_1 & 0 \\
    1 & 0 & k_{23} \\
    0 & k_{32} & 0
\end{bmatrix}.
\]
However, $\det Q = 0$, which is impossible, by Lemma~\ref{lem:equitable} and Corollary~\ref{cor:gt0}.

Next, we assume that both $\Gamma[V_2]$ and $\Gamma[V_3]$ are complete, that is, $k_{22} = n_2-1$ and $k_{33} = n_3-1$.
Since $\Gamma$ is triregular, we cannot have both $k_{21} = k_{31} = 1$.
Without loss of generality, since $\Gamma$ is connected, we assume that $k_{21}=1$ and $k_{31}=0$.
Furthermore, using \eqref{eqn:bbt22} and \eqref{eqn:bbt33}, we deduce that $M$ is the incidence matrix of a symmetric design with parameters $(v,k,\lambda)$, where $v = n_2 = n_3$, $k=k_{23}$, and $\lambda = k+(\theta_1+1)(\theta_2+1)$.
Thus, we obtain
\begin{align}
    (\theta_1+1)(\theta_2+1) & = \lambda-k. \label{eqn:t1mt23vcc}
\end{align}
The valency-partition $\mathfrak D(\Gamma)$, which, by Theorem~\ref{thm:equitableValency}, is equitable, has quotient matrix
\[
Q = Q_{\mathfrak D(\Gamma)}(\Gamma) = \begin{bmatrix}
    0 & v & 0 \\
    1 & v-1 & k \\
    0 & k & v-1
\end{bmatrix}.
\]
In this case $\det (xI-Q) = x^3-2(v-1)x^2+(v^2-k^2-3v+1)x+v(v-1)$.
Since $\det (Q) < 0$, using Lemma~\ref{lem:equitable}, it follows that $\operatorname{spec}Q = \{[\theta_0]^1,[\theta_1]^1,[\theta_2]^1 \}$.
Thus, we obtain
\begin{align}
    \theta_0 + \theta_1+\theta_2 &= 2(v-1);  \label{eqn:t0pt1pt2cc}\\
    \theta_0(\theta_1 + \theta_2)+\theta_1\theta_2 &= v^2-k^2-3v+1;  \label{eqn:et21cc}\\
    \theta_0 \theta_1 \theta_2 & = -v(v-1). \label{eqn:t0mt1mt2cc}
\end{align}
Putting \eqref{eqn:t1mt23vcc}, \eqref{eqn:t0pt1pt2cc}, \eqref{eqn:et21cc}, and \eqref{eqn:t0mt1mt2cc} together with \eqref{eqn:2desSYM} and the inequalities $\theta_0 > \theta_1 \geqslant 1$, and $\theta_2 \leqslant -2$ results in no solutions.

Lastly, we assume that $\Gamma[V_2]$ is empty and $\Gamma[V_3]$ is complete, i.e., $k_{22} = 0$ and $k_{33} = n_3-1$.
Furthermore, again using \eqref{eqn:bbt22} and \eqref{eqn:bbt33}, we deduce that $M$ is the incidence matrix of a symmetric design with parameters $(v,k,\lambda)$, where $v = n_2 = n_3$, $k=k_{23}$, and $\lambda = k+\theta_1\theta_2 = k+(\theta_1+1)(\theta_2+1)$.
Thus, we obtain
\begin{align}
    \theta_1+\theta_2 &= -1;  \label{eqn:t1pt23v}\\
    \theta_1 \theta_2 & = \lambda-k. \label{eqn:t1mt23v}
\end{align}

The valency-partition $\mathfrak D(\Gamma)$, which, by Theorem~\ref{thm:equitableValency}, is equitable, with quotient matrix
\[
Q = Q_{\mathfrak D(\Gamma)}(\Gamma) = \begin{bmatrix}
    0 & k_{21}v & k_{31}v \\
    k_{21} & 0 & k \\
    k_{31} & k & v-1
\end{bmatrix}.
\]
Since $\det Q = k_{21}v(1-v+2k_{31}k)$, by Lemma~\ref{lem:equitable} and Theorem~\ref{thm:chengbipartite}, we must have $k_{21} = 1$.
There remain two cases to consider: $k_{31} = 0$ and $k_{31} = 1$.

Consider the case $k_{31} = 0$, i.e.,
\[
Q = \begin{bmatrix}
    0 & v & 0 \\
    1 & 0 & k \\
    0 & k & v-1
\end{bmatrix}.
\]
Since $\det Q = -v(v-1) < 0$, by Lemma~\ref{lem:equitable}, we must have that $\theta_0$, $\theta_1$, and $\theta_2$ are eigenvalues of $Q$.
Thus, we obtain
\begin{align}
    \theta_0 + \theta_1+\theta_2 &= v-1;  \label{eqn:t0pt1pt}\\
    \theta_0 \theta_1 \theta_2 & = -v(v-1). \label{eqn:t0mt1mt}
\end{align}
Putting equations \eqref{eqn:t1pt23v}, \eqref{eqn:t1mt23v}, \eqref{eqn:t0pt1pt}, and \eqref{eqn:t0mt1mt} together with equation \eqref{eqn:2desSYM} and the inequalities $\theta_0 > \theta_1 \geqslant 1$, and $\theta_2 \leqslant -2$ results in no solutions.

Now, consider the case $k_{31} = 1$, i.e.,
\[
Q = \begin{bmatrix}
    0 & v & v \\
    1 & 0 & k \\
    1 & k & v-1
\end{bmatrix}.
\]
In this case $\det (xI-Q) = x^3-(v-1)x^2-(2v+k^2)x+v(v-1-2k)$.
Since $2v+k > 0$, using Lemma~\ref{lem:equitable}, it follows that either $\operatorname{spec}Q = \{[\theta_0]^1,[\theta_1]^1,[\theta_2]^1 \}$ or $\operatorname{spec}Q = \{[\theta_0]^1,[\theta_2]^2 \}$.
Thus, we obtain
\begin{align}
    \theta_0 + \theta^\prime+\theta_2 &= v-1;  \label{eqn:t0pt1pt2}\\
    \theta_0(\theta^\prime + \theta_2)+\theta^\prime\theta_2 &= -k^2-2v;  \label{eqn:et21}\\
    \theta_0 \theta^\prime \theta_2 & = v(2k-v+1) \label{eqn:t0mt1mt2}
\end{align}
corresponding to the two possible spectra of $Q$, corresponding to $\theta^\prime \in \{\theta_1,\theta_2\}$.
We can check for integer solutions for both of the resulting systems of equations.
For $\theta^\prime = \theta_1$, putting \eqref{eqn:t1pt23v}, \eqref{eqn:t1mt23v}, \eqref{eqn:t0pt1pt2}, \eqref{eqn:et21}, and \eqref{eqn:t0mt1mt2} together with \eqref{eqn:2desSYM} and the inequalities $\theta_0 > \theta_1 \geqslant 1$, and $\theta_2 \leqslant -2$ results in no solutions.
Finally, For $\theta^\prime = \theta_2$, putting \eqref{eqn:t1pt23v}, \eqref{eqn:t1mt23v}, \eqref{eqn:t0pt1pt2}, \eqref{eqn:et21}, and \eqref{eqn:t0mt1mt2} together with \eqref{eqn:2desSYM} and the inequalities $\theta_0 > \theta_1 \geqslant 1$, and $\theta_2 \leqslant -2$ results in no solutions.
\end{proof}

Now we are ready to prove Theorem~\ref{thm:rank14lb}.

\begin{proof}[Proof of Theorem~\ref{thm:rank14lb}]
    Suppose the rank of $\mathcal {W}(\Gamma)$ has rank less than $14$.
    By Lemma~\ref{lem:HigmanType} together with Lemma~\ref{lem:type211}, Lemma~\ref{lem:type31113}, and Lemma~\ref{lem:conesd}, $\mathcal {W}(\Gamma)$ must have type $ \left [\begin{smallmatrix} 1 & 1 & 1 \\  & 1 & 1 \\ &  & 1 \end{smallmatrix} \right ]$.
    It is straightforward to check that this is not possible.
\end{proof}

\begin{table}[htbp]
    \centering
    \begin{tabular}{c|c|c|c}
       Valencies & Spectrum & Coherent rank & Reference \\
       \hline
     {$\left \{[45]^1, [25]^{18}, [13]^{27} \right \}$}  & {$\left \{[21]^1,[3]^{19},[-3]^{26} \right \}$} & 16 & \cite{BM81}  \\
         {$\left \{[15]^{4}, [10]^{16}, [7]^{4} \right \}$}  & {$\left \{[11]^1,[3]^{7},[-2]^{16} \right \}$} & 18 &  \cite{VANDAM1998101}\\
    {$\left \{[96]^1, [61]^{64}, [21]^{32} \right \}$}  & {$\left \{[56]^1,[4]^{41},[-4]^{55} \right \}$} & 20 & \cite{BM81} \\
     $\left \{[24]^{18}, [14]^{9}, [8]^{9} \right \}$  & $\left \{[20]^1,[2]^{17},[-3]^{18} \right \}$ & 29 & \cite{vDddC05}\\
     {$\left \{[24]^{18}, [14]^{9}, [8]^{9} \right \}$}  & {$\left \{[20]^1,[2]^{17},[-3]^{18} \right \}$} & 240 & \cite{Cheng_2016} \\
     $\left \{[35]^{1}, [26]^{7}, [19]^{35} \right \}$  & $\left \{[21]^1,[\frac{-1\pm \sqrt{41}}{2}]^{21} \right \}$ & 949 & \cite{vDddC05}\\
     $\left \{[35]^{1}, [26]^{7}, [19]^{35}  \right \}$  & $\left \{[21]^1,[\frac{-1\pm \sqrt{41}}{2}]^{21} \right \}$ & 1849 & \cite{vDddC05}\\
  \end{tabular}
    \caption{List of known connected graphs that have three distinct eigenvalues and three distinct valencies.}
    \label{tab:3distinctValencies}
\end{table}

In Table~\ref{tab:3distinctValencies}, we list all currently known examples of connected graphs that have precisely three distinct eigenvalues and three distinct valencies.
We note that, among these graphs, the smallest coherent rank is $16$, which invites the question of whether the lower bound of Theorem~\ref{thm:rank14lb} can be improved.

In view of the following result, we denote by $\widehat{\Gamma}$ the cone over the graph $\Gamma$.

\begin{theorem}
    \label{thm:converse14}
    Let $\theta_1$ and $\theta_2$ be integers such that $\theta_1 \geqslant 1$ and $\theta_2 \leqslant -2$.
    Suppose $\mathcal Q$ is a QSD with parameters $(v,k,\lambda; b,r,\{x,y\})$ where $\mathsf B_x(\mathcal Q) \in \mathscr G_3(e_0,e_1,e_2)$ and
    \[
    (x,y) = (k+(\theta_1+1)(\theta_2+1) - (e_1+1)(e_2+1),k + \theta_1\theta_2 - e_1e_2).
    \]
    \begin{itemize}
        \item[(i)] If $\lambda = r + \theta_1\theta_2$, $\frac{r-\lambda+y-k}{x-y} = \theta_1+\theta_2$, and
        \begin{align*}
            k+e_0-\theta_1-\theta_2 &= \sqrt{(v+b+\theta_1\theta_2)(1+k+e_0+\theta_1\theta_2)} \\
            r-\theta_1-\theta_2 &= \sqrt{(v+b+\theta_1\theta_2)(1+\lambda)} \\
            1+k+\frac{\lambda k -yr}{x-y} &= \sqrt{(1+\lambda)(1+k+e_0+\theta_1\theta_2)}
        \end{align*}
        then $\widehat{\mathsf T_x(\mathcal Q)} \in \mathscr G_3 \left (\theta_0,\theta_1,\theta_2 \right )$ where $\theta_0$ is the largest eigenvalue of $\left [\begin{smallmatrix}
        0 & v & b \\ 1 & 0 & r \\ 1 & k & e_0
    \end{smallmatrix} \right ]$.
        Furthermore, if $|\{v+b,1+r, 1+k+e_0\}|=3$ then $\widehat{\mathsf T_x(\mathcal Q)}$ has coherent rank $14$.
        \item[(ii)] If $\lambda = r +(\theta_1+1)(\theta_2+1)$, $\frac{r-\lambda+y-k}{x-y} = \theta_1+\theta_2+1$, and
        \begin{align*}
            k+e_0-\theta_1-\theta_2 &= \sqrt{(v+b+\theta_1\theta_2)(1+k+e_0+\theta_1\theta_2)} \\
            v-1+r-\theta_1-\theta_2 &=  \sqrt{(v+b+\theta_1\theta_2)(v+r+\theta_1\theta_2)} \\
            1+k+\frac{\lambda k -yr}{x-y} &= \sqrt{(v+r+\theta_1\theta_2)(1+k+e_0+\theta_1\theta_2)}
        \end{align*}
         then $\widehat{\mathsf W_x(\mathcal Q)} \in \mathscr G_3 \left (\theta_0,\theta_1,\theta_2 \right )$, where $\theta_0$ is the largest eigenvalue of $\left [\begin{smallmatrix}
        0 & v & b \\ 1 & v-1 & r \\ 1 & k & e_0
    \end{smallmatrix} \right ]$.
    Furthermore, if $|\{v+b,v+r, 1+k+e_0\}|=3$ then $\widehat{\mathsf W_x(\mathcal Q)}$ has coherent rank $14$.
    \end{itemize}
\end{theorem}
\begin{proof}
The proofs that $\widehat{\mathsf T_x(\mathcal Q)}$ and  $\widehat{\mathsf W_x(\mathcal Q)}$ each have precisely three distinct eigenvalues is a modification of the proof of Theorem~\ref{thm:converse}, which we leave to the reader.
The assumptions also imply that $\widehat{\mathsf T_x(\mathcal Q)}$ and $\widehat{\mathsf W_x(\mathcal Q)}$ are triregular.
    By Theorem~\ref{thm:rank14lb}, the coherent rank of $\Gamma$ is at least $14$.
    Let $M$ be the incidence matrix of $\mathcal Q$.
      The upper bound of $14$ follows since $\mathcal {W}(\Gamma)$ is a subalgebra of the algebra
      \[
       \left \langle \begin{array}{c}
           \left [ \begin{smallmatrix}
          1 & \mathbf 0^\transpose & \mathbf 0^\transpose \\
          \mathbf 0 & O & O \\
          \mathbf 0 & O & O
      \end{smallmatrix} \right ], 
      \left [ \begin{smallmatrix}
          0 & \mathbf 1^\transpose & \mathbf 0^\transpose \\
          \mathbf 0 & O & O \\
          \mathbf 0 & O & O
      \end{smallmatrix} \right ],
       \left [ \begin{smallmatrix}
          0 & \mathbf 0^\transpose & \mathbf 1^\transpose \\
          \mathbf 0 & O & O \\
          \mathbf 0 & O & O
      \end{smallmatrix} \right ],
       \left [ \begin{smallmatrix}
          0 & \mathbf 0^\transpose & \mathbf 0^\transpose \\
          \mathbf 1 & O & O \\
          \mathbf 0 & O & O
      \end{smallmatrix} \right ],  \left [ \begin{smallmatrix}
          0 & \mathbf 0^\transpose & \mathbf 0^\transpose \\
          \mathbf 0 & O & O \\
          \mathbf 1 & O & O
      \end{smallmatrix} \right ],  \\
         \left [ \begin{smallmatrix}
          0 & \mathbf 0^\transpose & \mathbf 0^\transpose \\
          \mathbf 0 & I & O \\
          \mathbf 0 & O & O
      \end{smallmatrix} \right ], \left [ \begin{smallmatrix}
          0 & \mathbf 0^\transpose & \mathbf 0^\transpose \\
          \mathbf 0 & J-I & O \\
          \mathbf 0 & O & O
      \end{smallmatrix} \right ], 
      \left [ \begin{smallmatrix}
          0 & \mathbf 0^\transpose & \mathbf 0^\transpose \\
          \mathbf 0 & O & M \\
          \mathbf 0 & O & O
      \end{smallmatrix} \right ], \left [ \begin{smallmatrix}
          0 & \mathbf 0^\transpose & \mathbf 0^\transpose \\
          \mathbf 0 & O & J-M \\
          \mathbf 0 & O & O
      \end{smallmatrix} \right ], \left [ \begin{smallmatrix}
          0 & \mathbf 0^\transpose & \mathbf 0^\transpose \\
          \mathbf 0 & O & O \\
          \mathbf 0 & M^\transpose & O
      \end{smallmatrix} \right ], 
      \left [ \begin{smallmatrix}
          0 & \mathbf 0^\transpose & \mathbf 0^\transpose \\
          \mathbf 0 & O & O \\
          \mathbf 0 & J-M^\transpose & O
      \end{smallmatrix} \right ], \\
      \left [ \begin{smallmatrix}
          0 & \mathbf 0^\transpose & \mathbf 0^\transpose \\
          \mathbf 0 & O & O \\
          \mathbf 0 & O & I
      \end{smallmatrix} \right ],  \left [ \begin{smallmatrix}
          0 & \mathbf 0^\transpose & \mathbf 0^\transpose \\
          \mathbf 0 & O & O \\
          \mathbf 0 & O & A(\mathsf B_x(\mathcal Q))
      \end{smallmatrix} \right ],  \left [ \begin{smallmatrix}
          0 & \mathbf 0^\transpose & \mathbf 0^\transpose \\
          \mathbf 0 & O & O \\
          \mathbf 0 & O & J-I-A(\mathsf B_x(\mathcal Q))
      \end{smallmatrix} \right ]
      \end{array}\right \rangle,
      \]
      which is a coherent algebra of rank $14$.
\end{proof}

A cursory search of possible QSD parameters that satisfy Theorem~\ref{thm:converse14} resulted in just one possibility.

  \begin{example}
  \label{ex:14}
      Suppose $\mathcal  Q$ is a QSD with parameters $(85,35,34; 204,84,\{10,15\})$.
      Then $\mathcal Q$ satisfies Theorem~\ref{thm:converse14} with $\theta_1 = 4$ and $\theta_2 = -11$.
      Furthermore, by Theorem~\ref{thm:converse14}, we have $\widehat{\mathsf W_{10}(\mathcal Q)} \in \mathscr G_3(119,4,-11)$.
      In particular, $\widehat{\mathsf W_{10}(\mathcal Q)}$ has spectrum $\{[119]^1,[4]^{204},[-11]^{85}\}$, its degree sequence  is $\{[289]^1,[169]^{85},[64]^{204}\}$ and its coherent rank is $14$.
  \end{example}

  The existence of a QSD corresponding to Example~\ref{ex:14} is an open problem (see Question~\ref{q:14}).

\section{Large coherent rank}
\label{sec:latin}

In the previous sections, we have been concerned with graphs in $\mathscr G_3$ that have small coherent rank.
We now turn our attention to large coherent rank.
Before our work, the largest known coherent rank of a graph in $\mathscr G_3$ has coherent rank at most $19^6$.
This graph, which was discovered by Van Dam~\cite[Section 2.2]{VANDAM1998101} was obtained by taking a strongly regular graph on $19^4$ vertices and interchanging some of its edges with nonedges (see \emph{switching}, below).
Furthermore, all known infinite families of graphs in $\mathscr G_3$ have a fixed coherent rank.

In this section, we exhibit a conjecturally infinite family of biregular graphs in $\mathscr G_3$ (see Remark~\ref{rem:infinitefamily}) whose coherent ranks can conjecturally become arbitrarily large (see Question
~\ref{qn:largerank}).

\subsection{Switching strongly regular graphs}

Let $\Gamma = (V,E)$ be a graph and let $\sigma = \{U,V-U\}$ be a partition of the vertex set $V$.
We can write the adjacency matrix of $\Gamma$ in block form as
\[
A(\Gamma) = \begin{bmatrix}
    A(\Gamma[U]) & M \\
    M^\transpose & A(\Gamma[V-U])
\end{bmatrix}.
\]
Define the (switched) graph $\Gamma^\sigma$ via its adjacency matrix as
\[
A(\Gamma^\sigma) = \begin{bmatrix}
    A(\Gamma[U]) & J-M \\
    J-M^\transpose & A(\Gamma[V-U])
\end{bmatrix}.
\]
We say that the graph $\Gamma^\sigma$ was obtained by \textbf{switching} the graph $\Gamma$ with respect to the vertex-partition $\sigma$.
We refer to $\Gamma^\sigma$ as the \textbf{switched} graph.
Muzychuk and Klin~\cite{MUZYCHUK1998191} used the following tool to study the spectrum of a graph obtained by switching some other graph.

\begin{proposition}[{\cite[Corollary 3.2]{MUZYCHUK1998191}}]
\label{pro:specSwitch}
    Let $\sigma$ be an equitable $2$-partition of a graph $\Gamma$.
    Then
    \[
    \operatorname{spec} \left (A(\Gamma^{\sigma}) \right ) =  \operatorname{spec}(A(\Gamma)) \cup \operatorname{spec}(Q_\sigma(\Gamma^{\sigma})) - \operatorname{spec}(Q_\sigma(\Gamma)) .
    \]
\end{proposition}

In view of Proposition~\ref{pro:specSwitch}, to find a graph with three distinct eigenvalues, it suffices to find a strongly regular graph $\Gamma = (V,E)$ and an equitable partition $\sigma = \{U,V-U\}$ such that $\operatorname{spec}(Q_\sigma(\Gamma^\sigma))$ shares an eigenvalue with $\Gamma$.
This is how we proceed with a focus on the family of block graphs of orthogonal arrays.

\subsection{Block graphs of orthogonal arrays}

For $m \geqslant 2$ and $n \geqslant 1$, an \textbf{orthogonal array} $\operatorname{OA}(m,n)$ is an $m \times n^2$ matrix $M$ with entries from the set $\{1,\dots,n\}$ such that the $n^2$ columns of each $2 \times n^2$ submatrix of $M$ contain all $n^2$ ordered pairs of elements of $\{1,\dots,n\}$.

\begin{theorem}[{\cite[Theorems 6.39 and 6.40]{stinson}}]
    \label{thm:existence}
    Let $n$ be a prime power and $2 \leqslant m \leqslant n+1$.
    Then there exists an $\operatorname{OA}(m,n)$.
\end{theorem}

We will only consider the construction implicit from Theorem~\ref{thm:existence} (see also~\cite[Construction 3.29]{handbookCD}).  
However, there exist other constructions of orthogonal arrays \cite[Section 6.5]{stinson}. 
We denote by $\mathsf{OA}(m,n)$ the graph whose vertices are the columns of $\operatorname{OA}(m,n)$ where two columns are adjacent if there is a row where they share the same entry.
The graph $\mathsf{OA}(m,n)$ is commonly known as the \textbf{block graph of an orthogonal array} and is known as the \emph{Latin squares graph} in \cite{VANDAM1998101}, therein denoted by $L_m(n)$.

\begin{theorem}[{\cite[Section 5.5]{GM}}]
\label{thm:OAeigs}
    Let  $m \geqslant 2$ and $n \geqslant 1$ be integers satisfying $m \leqslant n+1$.
    Then the graph $\mathsf {OA}(m,n)$ has spectrum
    \[
    \left \{ [m(n-1)]^1, [n-m]^{m(n-1)}, [-m]^{(n-1)(n+1-m)} \right \}.
    \]
\end{theorem}

Given an orthogonal array $\operatorname{OA}(m,n)$ and $i \in \{1,\dots,n\}$, let $C_i$ be the set of columns of $\operatorname{OA}(m,n)$ that have the entry $1$ in the $i$th row.
Then $C_i$ is a clique of order $n$ in the graph $\mathsf{OA}(m,n)$.
Furthermore, the cliques $C_i$ and $C_j$ are vertex disjoint for any $i \ne j$.

\subsection{A new family of biregular graphs with three distinct eigenvalues}

For $m \geqslant 2$ and $n \geqslant 1$, let $\Gamma = \mathsf{OA}(m,n)$ and let $\sigma = \{U,V-U\}$ be a partition of the vertex set $V(\Gamma)$ where $U = C_1 \cup \dots \cup C_N$.
Denote by $\mathsf S_N(m,n)$ the switched graph $\Gamma^\sigma$.

\begin{lemma}
\label{lem:N}
    Let $n \in \mathbb N$ and $2 \leqslant m \leqslant n+1$.
    Suppose $1 \leqslant N \leqslant n-1$.
    Then $\mathsf S_N(m,n)$ has three distinct eigenvalues if and only if
    \[
    N = \frac{n}{2} \pm \frac{\sqrt{(n-2m)(n-2m+2)(n+2)n}}{2(n-2m+2)}.
    \]
    Moreover, $\mathsf S_N(m,n)$ has spectrum
\[
    \left \{ [n+m(n-1)]^1, [n-m]^{m(n-1)-1}, [-m]^{(n-1)(n+1-m)+1} \right \}.
    \]
\end{lemma}
\begin{proof}
The quotient matrix
\[
Q_\sigma(\Gamma) = 
\begin{bmatrix}
    n-1 + (N-1)(m-1) & (n-N)(m-1) \\ N(m-1) & n-1 + (n-N-1)(m-1)
\end{bmatrix}
\]
has eigenvalues $m(n-1)$ and $n-m$.
Using Theorem~\ref{thm:OAeigs} and Proposition~\ref{pro:specSwitch}, for $\mathsf S_N(m,n)$ to have three distinct eigenvalues, we require $-m$ to be an eigenvalue of the quotient matrix 
\[
Q_\sigma(\Gamma^\sigma) = 
\begin{bmatrix}
    n-1 + (N-1)(m-1) & (n-N)(n-m+1) \\ N(n-m+1) & n-1 + (n-N-1)(m-1)
\end{bmatrix}.
\]
It is straightforward to check that $-m$ is a eigenvalue of $Q_\sigma(\Gamma^\sigma)$ if and only if
\[
N = \frac{n}{2} \pm \frac{\sqrt{(n-2m)(n-2m+2)(n+2)n}}{2(n-2m+2)}.
\]
For such $N$, the quotient matrix $Q_\sigma(\Gamma^\sigma)$ has eigenvalues $n+m(n-1)$ and $-m$.
\end{proof}

In view of Lemma~\ref{lem:N}, for $m \geqslant 2$ and $n \geqslant 1$ where $N = \frac{n}{2} - \frac{\sqrt{(n-2m)(n-2m+2)(n+2)n}}{2(n-2m+2)}$ is a positive integer less than $n$, we define the graph $\mathsf G(m,n) := \mathsf S_N(m,n)$.
Clearly, in order for $\mathsf G(m,n)$ to exist, we require the existence of an orthogonal array $\operatorname{OA}(m,n)$.
We now aim to find integers $m$ and $n$ such that $\mathsf G(m,n)$ exists.

\begin{lemma}
    Suppose that $\mathsf G(m,n)$ exists for some $n \geqslant 2$ and $m$.
    Then $\frac{n}{3} < m \leqslant \frac{n}{2}$.
\end{lemma}

\begin{proof}
The proof follows from the condition that $N = \frac{n}{2} - \frac{\sqrt{(n-2m)(n-2m+2)(n+2)n}}{2(n-2m+2)}$ is a positive integer less than $n$.
Suppose $n - 2m < -2$.
Then 
\[
-n(n-2m+2) < \sqrt{(n-2m)(n-2m+2)(n+2)n},
\]
which implies that $N \geqslant n$, a contradiction.
Hence, $n - 2m > -2$.
Furthermore, since $(n-2m)(n-2m+2)(n+2)n \geqslant 0$, we must have $n - 2m \geqslant 0$.
Now, using the requirement that $N \geqslant 1$, we obtain $3m \geqslant n + \frac{5n-6}{3n-2}$, as required.
\end{proof}

Note that when $n = 2m$, the graph $\mathsf G(m,n)$ is regular, and hence, strongly regular.
Now, we will consider the special case when $n = 2m+1$.
In this case, using Theorem~\ref{thm:existence}, to show that $\mathsf G(\frac{n-1}{2},n)$ exists, it suffices to take $n$ to be a prime power such that $n/2 - \sqrt{3n(n + 2)}/6$ is an integer.
Let $(a_k)_{k \in \mathbb N}$ be the recurrence sequence defined by $a_k = 4a_{k-1} - a_{k-2}$ with initial conditions $a_0=1$, $a_1 = 5$.

\begin{lemma}
\label{lem:oddSquare}
    For each $k \in \mathbb N$, we have that $3(a_k^2+2) = (a_{k+1}-2a_k)^2$ is an odd square.
\end{lemma}
\begin{proof}
    It is routine to check that $(a_{k+1}-2a_k+a_k\sqrt{3})(a_{k+1}-2a_k-a_k\sqrt{3})=6$, as required.
\end{proof}

If $n = a_k^2$ for some $k \in \mathbb N$ then, by Lemma~\ref{lem:oddSquare}, $\frac{n}{2} - \frac{\sqrt{3n(n + 2)}}{6}$ is an integer.
Thus, using Theorem~\ref{thm:existence}, each prime power in the sequence $(a_k)_{k \in \mathbb N}$ can produce a biregular graph with three distinct eigenvalues.

\begin{corollary}
  \label{cor:3evsConstruct}
  Let $k \in \mathbb N$.
  Suppose that $q=a_k$ is a prime power.
  Then $\mathsf G\left (\frac{q^2-1}{2},q^2 \right)$ is biregular and has spectrum
  \[
    \left \{ \left [\frac{q^4+1}{2} \right ]^1, \left [\frac{q^2+1}{2} \right ]^{\frac{q^4-2q^2-1}{2}}, \left [\frac{1-q^2}{2}\right ]^{\frac{q^4+2q^2-1}{2}} \right \}.
    \]
\end{corollary}

There are only finitely many terms of the form $p^e$ in the sequence $(a_k)_{k \in \mathbb N}$ where $e > 1$ and $p$ is a prime~\cite{Stewart83}.
The (probable) primes in the sequence $(a_k)_{k \in \mathbb N}$ are listed in OEIS sequence A299107~\cite{oeis}.

\begin{remark}
\label{rem:infinitefamily}
    It is conjectured \cite[Conjecture 47]{Hone18} that $(a_k)_{k \in \mathbb N}$ contains infinitely many primes.
By Corollary~\ref{cor:3evsConstruct}, the validity of this conjecture implies the existence of a new infinite family of biregular graphs with three distinct eigenvalues: $\mathsf G \left (\frac{b_k^2-1}{2},b_k^2 \right )$ where $(b_k)_{k \in \mathbb N}$ is the subsequence of primes in $(a_k)_{k \in \mathbb N}$.
\end{remark}

Note that Van Dam~\cite{VANDAM1998101} already discovered two instances of Corollary~\ref{cor:3evsConstruct}, that is, $\mathsf G(12,25)$ and $\mathsf G(180,361)$.
In the same paper, Van Dam also discovered the graphs $\mathsf G(7,16)$ and $\mathsf G(120,243)$.
One can show, using elementary number theoretic arguments, that $\mathsf G(7,16)$ is the only graph of the form $\mathsf G(2^{k-1}-1,2^k)$.
We leave this to the reader to verify.

\section{Conclusion and open problems}
\label{sec:open}
We conclude with a selection of questions that arise naturally from our investigations.

\begin{question}
\label{qn:largerank}
    Can graphs in $\mathscr G_3(\theta_0,\theta_1,\theta_2)$ have arbitrarily large coherent rank?
\end{question}

All currently known infinite families in $\mathscr G_3(\theta_0,\theta_1,\theta_2)$ have a fixed coherent rank.
However, we have empirical evidence that suggests that the coherent rank of the graphs $\mathsf G(m,n)$ grows with $n$.
In particular, $\mathsf G(7,16)$ has coherent rank 2048 and $\mathsf G(12,25)$ has coherent rank $25^3 = 15625$.
Note that we are assuming the construction implicit in Theorem~\ref{thm:existence} - alternative constructions of orthogonal arrays could potentially produce different coherent ranks.
We conjecture that if an infinite family of graphs $\mathsf G(m,n)$ exists then the answer to Question~\ref{qn:largerank} is yes.

\begin{question}
    Does there exist $\Gamma \in \mathscr G_3(\theta_0,\theta_1,\theta_2)$ such that $\mathcal {W}(\Gamma)$ has rank $11$?
\end{question}

By Theorem~\ref{thm:muzychuk} together with Example~\ref{ex:rank8} and Example~\ref{ex:rank9}, we observe that there exist infinite families of graphs in $\mathscr G_3(\theta_0,\theta_1,\theta_2)$ having coherent ranks $3$, $5$, $6$, $8$, and $9$.
We know that no graph in $\mathscr G_3(\theta_0,\theta_1,\theta_2)$ can have coherent rank $4$ or $7$, by Theorem~\ref{thm:muzychuk} and Theorem~\ref{thm:wl7}.
Van Dam gave a construction of an infinite family of graphs in  $\mathscr G_3(\theta_0,\theta_1,\theta_2)$ that have coherent rank $10$, using a certain combination of symmetric designs and distance regular graphs~\cite[Section 6.1]{VANDAM1998101}. 
There are currently no known constructions of graphs in $\mathscr G_3(\theta_0,\theta_1,\theta_2)$ that have coherent rank $11$.
In Table~\ref{tab:rankExist}, we list the known realisable coherent ranks from graphs in $\mathscr G_3(\theta_0,\theta_1,\theta_2)$.

\begin{table}[htbp]
    \centering
    \begin{tabular}{c|cccccccccccccccccc}
        $r$ & 3 & 4 & 5 & 6 & 7 & 8 & 9 & 10 & 11 & 12 & 13 & 14 & 15 & 16 & 17 & 18 & 19 & 20 \\
        \hline
        Exists & Y & N & Y & Y & N & Y & Y & Y & ? & Y & Y & ? & ? & Y & ? & Y & ? & Y
    \end{tabular}
    \caption{The existence of graphs in $\mathscr G_3(\theta_0,\theta_1,\theta_2)$ having coherent rank $r$ for $3 \leqslant r \leqslant 20$.}
    \label{tab:rankExist}
\end{table}

Coherent ranks $16$, $18$, and $20$ correspond to triregular graphs in Table~\ref{tab:3distinctValencies}.
We also have examples of graphs in $\mathscr G_3(\theta_0,\theta_1,\theta_2)$ that have coherent ranks $12$ and $13$.

\begin{example}
    \label{ex:12}
    Let $\Gamma \in \operatorname{SRG}(40,12,2,4)$.
    The cardinality of $\operatorname{SRG}(40,12,2,4)$ is 28~\cite{Srg40}.
    Fix a vertex $\mathsf v \in V(\Gamma)$ such that $N(\mathsf v)$ is the disjoint union of four copies of $K_3$.
    Now partition $N(\mathsf v) = U \cup W$ such that $U$ and $W$ are disjoint subsets of cardinality $6$ and there are no edges between vertices in $U$ and vertices in $W$.
    Obtain a new graph on $39$ vertices by adding all edges between vertices of $U$ and $W$ and deleting the vertex $\mathsf v$. 
    The resulting graph $\Delta$ has degree sequence $\left \{ [12]^{27},[17]^{12} \right \}$ and spectrum $\{ [14]^1, [2]^{23},[-4]^{15} \}$.
    This method can produce $55$ pairwise non-isomorphic graphs of which $27$ have coherent rank $12$.

    Moreover, each graph can be switched with respect to its valency partition to produce a graph with degree sequence $\left \{ [16]^{27},[26]^{12} \right \}$ and spectrum $\{ [20]^1, [2]^{22},[-4]^{16} \}$.
    The coherent rank of switched graphs obtained from graphs with coherent rank $12$ remains equal to $12$.
    The type of the coherent closure of each of these graphs is $\left [ \begin{smallmatrix}
        4 & 2 \\
         & 4
    \end{smallmatrix} \right ]$.
\end{example}

\begin{example}
    \label{ex:13}
    Let $\mathcal D$ be a $2$-$(45,12,3)$ design that possesses a polarity with $36$ absolute points, that is, $\mathcal D$ has a symmetric incidence matrix $M$ whose diagonal contains $36$ entries equal to $1$.
    Suppose that, in addition, the $9 \times 9$ principal submatrix $N$ of $M$ induced on the rows/columns with a 0 on the diagonal is equal to the zero matrix $O$.
    Replace all the entries of $M$ in this $9 \times 9$ submatrix by $1$s.
    Now, form the adjacency matrix $A$ by setting its off-diagonal entries equal to those of $M$ and its diagonal entries equal to $0$.
    The resulting graph whose adjacency matrix is $A$ has degree sequence $\left \{ [11]^{36},[20]^{9} \right \}$ and spectrum $\{ [14]^1, [2]^{27},[-4]^{17} \}$.
    Using the list of $2$-$(45,12,3)$ designs from~\cite{Sanja}, we can produce nine pairwise non-isomorphic graphs of which four have coherent rank $13$.

    Moreover, each graph can be switched with respect to its valency partition to produce a graph with degree sequence $\left \{ [14]^{36},[32]^{9} \right \}$ and spectrum $\{ [20]^1, [2]^{26},[-4]^{18} \}$.
    The coherent rank of switched graphs obtained from graphs with coherent rank $13$ remains equal to $13$.
    The type of the coherent closure of each of these graphs is $\left [ \begin{smallmatrix}
        2 & 2 \\
         & 7
    \end{smallmatrix} \right ]$.
\end{example}

\begin{question}
    Does there exist $\Gamma \in \mathscr G_3$ such that $\mathcal {W}(\Gamma)$ has type $\left [ \begin{smallmatrix}
        3 & 2 \\
         & 3
    \end{smallmatrix} \right ]$?
\end{question}

Both constructions of Van Dam~\cite[Section 6.1 and Section~6.2]{VANDAM1998101} produce graphs whose coherent closure have type $\left [ \begin{smallmatrix}
        2 & 2 \\
         & 4
    \end{smallmatrix} \right ]$.
Coherent configurations that have type $\left [ \begin{smallmatrix}
        3 & 2 \\
         & 3
    \end{smallmatrix} \right ]$ are known as \emph{strongly regular designs}~\cite{srds}.
    We do not know if any graphs in $\mathscr G_3$ have a coherent closure equal to the adjacency algebra of a strongly regular design.

\begin{question}
\label{q:14}
    Does there exist a QSD with parameters $(85,35,34; 204,84,\{10,15\})$?
\end{question}

Using Example~\ref{ex:14}, the existence of such a QSD would show that the bound in Theorem~\ref{thm:rank14lb} is sharp.

\begin{question}
    Does there exist a graph in $\Gamma \in \mathscr G_3$ that has coherent rank $9$ and class $4$?
\end{question}

We have examples of graphs in $\mathscr G_3$ having coherent rank $9$ in classes 1, 2, and 3.
However, we are not aware of any example of a graph in class 4.

\begin{question}
    Does there exist a graph in $\Gamma \in \mathscr G_3$ that satisfies Lemma~\ref{lem:case3b}?
\end{question}

Example~\ref{ex:MK} is an example of a graph in class 3.
This graph is in $\mathscr G_3(21,5,-2)$ and its smallest eigenvalue $-2$ is an eigenvalue of the quotient matrix of its valency partition.
\textit{A priori}, it is possible for the second largest eigenvalue of a graph of class 3 to be an eigenvalue of the quotient matrix of its valency partition.
Such graphs must satisfy the assumption of Lemma~\ref{lem:case3b}.
However, we are not aware of any such example.

\section{Acknowledgements}

The first author is grateful to Saveliy Skresanov for informing us about the stabilization routine due to Sven Reichard~\cite{Sven}, to Edwin van Dam, Bill Martin, Misha Klin and Misha Muzychuk for their comments and information about the history of Haemer's question.
We have also benefited from insightful conversations with Ian Wanless and Vedran Kr\v cadinac.

The first author was partially supported by the Singapore Ministry of Education Academic Research Fund; grant numbers: RG18/23 (Tier 1) and MOE-T2EP20222-0005 (Tier 2).

\bibliographystyle{amsplain}
\providecommand{\bysame}{\leavevmode\hbox to3em{\hrulefill}\thinspace}
\providecommand{\MR}{\relax\ifhmode\unskip\space\fi MR }
\providecommand{\MRhref}[2]{%
  \href{http://www.ams.org/mathscinet-getitem?mr=#1}{#2}
}
\providecommand{\href}[2]{#2}

\appendix

\section{Tables of feasible parameters for QSDs}

In the appendix we provide tables of QSD parameters that correspond to the lemmas in Section~\ref{sec:param9}.
Together with Theorem~\ref{thm:CHNineq}, we use following necessary condition on the parameters of a QSD.

\begin{theorem}[{\cite[Theorem 48.13]{shrikhandeHCD}}]
    \label{thm:abs}
   Let $\mathcal Q$ be a QSD with parameters $(v,k,\lambda;b,r,\{x,y\})$ without repeated blocks.
   Then $b \leqslant \binom{v}{2}$, with equality if and only if $(\mathcal P, \mathcal B)$ is a $4$-design.
\end{theorem}

Calderbank~\cite{CALDERBANK198794,CALDERBANK1988} established additional necessary conditions on QSD parameters, which we refer to as C1 and C2 indicated in Theorem~\ref{thm:A}.
Let $\mathbb F_p$ denote the finite field with $p$ elements and let $S_p$ denote the subset of squares in $\mathbb F_p$.

\begin{theorem}
\label{thm:A}
    Let $\mathcal Q$ be a QSD with parameters $(v,k,\lambda;b,r,\{x,y\})$ and intersection numbers $x$ and $y$.
\begin{itemize}
    \item[(C1)] Suppose $x \equiv y \pmod 2$ and $r \not \equiv \lambda \pmod 4$.
    Then either
    \begin{align*}
        x &\equiv 0 \pmod 2, \quad  k \equiv 0 \pmod 4, \text { and }  v \equiv \pm 1 \pmod 8, \text{ or } \\
        x &\equiv 1 \pmod 2, \quad k \equiv v \pmod 4, \text { and }  v \equiv \pm 1 \pmod 8.
    \end{align*}
    \item[(C2)] Suppose $x \equiv y \pmod p$, where $p$ is an odd prime and $r \not \equiv \lambda \pmod {p^2}$.
    Then either
    \begin{itemize}
        \item $v \equiv 0 \pmod 2, \quad  v \equiv k \equiv x \equiv 0 \pmod p, \text { and }  (-1)^{v/2} \in S_p \text{ or }$
        \item $v \equiv 1 \pmod 2, \quad  v \equiv k \equiv x \not \equiv 0 \pmod p, \text { and }  x(-1)^{(v-1)/2} \in S_p;$
        \item $\lambda \equiv r \equiv 0 \pmod p$ and either
        \begin{align*}
        v &\equiv 0 \pmod 2, \quad \text{and} \quad  v \equiv k \equiv x \not \equiv 0 \pmod p; \\
        v &\equiv 0 \pmod 2, \quad  k \equiv x \not \equiv 0 \pmod p, \text { and }  v/x \not \in S_p; \\
        v &\equiv 1 \pmod {2p}, \quad  r \equiv 0 \pmod {p^2}, \text { and }  k \equiv x \not \equiv 0 \pmod p; \\
        v &\equiv p \pmod {2p}, \quad \text{and} \quad  k \equiv x \equiv 0 \pmod p; \\
        v &\equiv 1 \pmod 2, \quad   k \equiv x \equiv 0 \pmod p, \text { and }  v \not \in S_p; \\
        v &\equiv 1 \pmod 2, \quad  k \equiv x \equiv 0 \pmod p, v \in S_p \text{ and } (-1)^{(v-1)/2}\in S_p. \\
    \end{align*}
    \end{itemize}
\end{itemize}
\end{theorem}

In the first column of the tables below, we list parameters of QSD corresponding to Class 1, Class 2, Class 3, and Class 4 graphs from Section~\ref{sec:param9}.
We only list parameters that satisfy both Theorem~\ref{thm:CHNineq} and Theorem~\ref{thm:abs}.
In the second column, we write the three distinct eigenvalues of the resulting total graph or whole graph.
The third column indicates the existence of a QSD having the given parameters.
If the nonexistence of a QSD with given parameters is due to a violation of C1 or C2 from Theorem~\ref{thm:A} then this is indicated in the fourth column, together with the relevant prime in the case when C2 is violated.
If a corresponding QSD exists, further details are given in the fourth column.

\begin{table}[htp]
    \centering
    \scriptsize
    \begin{tabular}{ |r|l|c|c|}
        \hline
        $(v,k,\lambda;b,r,\{x,y\}$)  & $(\theta_0,\theta_1,\theta_2)$ & Exists & Remark  \\ \hline 
${(225, 36, 10; 400, 64,\{0,6\})}$ & ${(384, 9, -6)}$ & {?} & {} \\
${(232, 36, 15; 638, 99,\{0,6\})}$ & ${(594, 14, -6)}$ & {N} & {C2 $(p=3)$} \\
${(3501, 225, 40; 9725, 625,\{0,15\})}$ & ${(9375, 39, -15)}$ & {?} & {} \\
${(64576, 1600, 65; 105945, 2625,\{0,40\})}$ & ${(105000, 64, -40)}$ & {N} & {C2 $(p=5)$} \\
${(1075, 100, 66; 7697, 716,\{0,10\})}$ & ${(7160, 65, -10)}$ & {N} & {C1} \\

${(9549, 441, 70; 32891, 1519,\{0,21\})}$ & ${(31899, 69, -21)}$ & {?} & {} \\
${(2891, 196, 78; 17051, 1156,\{0,14\})}$ & ${(16184, 77, -14)}$ & {N} & {C1} \\
${(43561, 1225, 85; 107569, 3025,\{0,35\})}$ & ${(105875, 84, -35)}$ & {?} & {} \\
${(3550, 225, 96; 23998, 1521,\{0,15\})}$ & ${(22815, 95, -15)}$ & {N} & {C2 $(p=3)$} \\
       \hline
    \end{tabular}
    \caption{List of QSD parameters corresponding to Lemma~\ref{lem:case1} where $1\leqslant \theta_1 \leqslant 100$.}
    \label{tab:paramC1}
\end{table}
\begin{table}[htp]
    \centering
    \scriptsize
\begin{tabular}{ |r|l|c|c|}
        \hline
        $(v,k,\lambda;b,r,\{x,y\})$  & $(\theta_0,\theta_1,\theta_2)$ & Exists & Remark  \\ \hline 
        ${(121, 46, 69; 484, 184, \{16, 21\})}$ & ${(368, 5, -23)}$ & {?} & {} \\
${(172, 64, 112; 817, 304, \{22,28\})}$ & ${(608, 6, -32)}$ & {N} & {C2 $(p=3)$} \\
${(661, 177, 236; 3305, 885, \{45,56\})}$ & ${(2655, 11, -59)}$ & {?} & {} \\
${(379, 136, 340; 2653, 952, \{46, 55\})}$ & ${(1904, 9, -68)}$ & {?} & {} \\
${(361, 145, 522; 3249, 1305, \{55,64\})}$ & ${(2175, 9, -87)}$ & {?} & {} \\
${(3627, 540, 330; 14911, 2220, \{78,99\})}$ & ${(13320, 21, -90)}$ & {N} & {C2 $(p=7)$} \\
       \hline
    \end{tabular}
    \caption{List of QSD parameters corresponding to Lemma~\ref{lem:case2} where $-100\leqslant \theta_2\leqslant -2$.}
    \label{tab:paramC2}
\end{table}
\begin{table}[htp]
    \centering
    \scriptsize
    \begin{tabular}{ |r|l|c|c|}
        \hline
        $(v,k,\lambda;b,r,\{x,y\})$  & $(\theta_0,\theta_1,\theta_2)$ & Exists & Remark  \\ \hline 
        ${(8, 6, 15; 28, 21, \{4, 5\})}$ & ${(21, 5, -2)}$ & {Y} & {Example~\ref{ex:MK}} \\
        ${(120, 50, 35; 204, 85, \{20,25\})}$ & ${(153, 9, -6)}$ & {?} & {} \\
        ${(76, 40, 52; 190, 100, \{20, 24\})}$ & ${(125, 11, -5)}$ & {?} & {} \\
${(120, 75, 370; 952, 595, \{45,50\})}$ & ${(476, 44, -6)}$ & {?} & {} \\
${(169, 105, 585; 1521, 945, \{63, 69\})}$ & ${(735, 59, -7)}$ & {?} & {} \\
       \hline
    \end{tabular}
    \caption{List of QSD parameters corresponding to Lemma~\ref{lem:case3a} where $1 \leqslant \theta_1 \leqslant 100$.}
    \label{tab:paramC3}
\end{table}
\begin{table}[h!tp]
    \centering
    \scriptsize
    \begin{tabular}{ |r|l|c|c|}
        \hline
        $(v,k,\lambda;b,r,\{x,y\})$  & $(\theta_0,\theta_1,\theta_2)$ & Exists & Remark  \\ \hline 
        
${(141, 45, 33; 329, 105, \{9, 15\})}$ & ${(175, 5, -13)}$ & {?} & {} \\
${(85, 40, 130; 595, 280, \{15,20\})}$ & ${(224, 4, -31)}$ & {?} & {} \\
${(232, 112, 296; 1276, 616, \{48,56\})}$ & ${(539, 7, -41)}$ & {?} & {} \\
${(5866, 1666, 777; 9637, 2737, \{441, 476\})}$ & ${(6647, 34, -57)}$ & {N} & {C2 $(p=5)$} \\
${(3655, 1450, 1380; 8772, 3480, \{550,580\})}$ & ${(5046, 29, -71)}$ & {N} & {C2 $(p=3)$} \\
       \hline
    \end{tabular}
    \caption{List of QSD parameters corresponding to Lemma~\ref{lem:case4b} where $-100 \leqslant \theta_2 \leqslant -2$.}
    \label{tab:paramC4}
\end{table}
\end{document}